\documentclass{amsart}
\usepackage{amsmath,enumerate,amsfonts,picture,graphicx}

\newtheorem{lemma}{Lemma}[section]

\newtheorem{proposition}[lemma]{Proposition}
\newtheorem{theorem}[lemma]{Theorem}
\newtheorem{corollary}[lemma]{Corollary}

\newtheorem{example}[lemma]{Example}

\newcommand{\R}{\mathbb{R}}
\newcommand{\RP}{\mathbb{RP}}

\newcommand{\jj}{\mathtt{j}}
\newcommand{\ii}{\mathtt{i}}
\newcommand{\kk}{\mathtt{k}}
\newcommand{\lll}{\mathtt{l}}
\newcommand{\e}{\varepsilon}

\newcommand{\adim}{\dim_{\mathsf{AFF}}}
\newcommand{\ldim}{\dim_{\mathsf{LY}}}

\title[On equality of Hausdorff and affinity dimensions]{On equality of Hausdorff and affinity dimensions, via self-affine measures on positive subsystems}
\author{Ian D. Morris}

\address{Department of Mathematics, University of Surrey, Guildford GU2 7XH, United Kingdom.}

\urladdr{http://personal.maths.surrey.ac.uk/st/I.Morris}
\email{i.morris@surrey.ac.uk}

\author{Pablo Shmerkin}

\address{Departamento de Matem\'{a}ticas y Estad\'{\i}sticas and CONICET\\
Universidad Torcuato Di Tella\\
Av. Figueroa Alcorta 7350 (C1428BCW), Buenos Aires, Argentina.}

\urladdr{http://www.utdt.edu/profesores/pshmerkin}
\email{pshmerkin@utdt.edu}

\subjclass[2000]{28A80, 37C45 (primary) 37D35 (secondary)}

\thanks{Ian Morris was supported by the Engineering and Physical Sciences Research Council (grant number EP/L026953/1). Pablo Shmerkin was partially supported by project PICT 2013-1393 (ANPCyT). Ian Morris also wishes to thank Thomas Kempton for helpful conversations, and Universidad Torcuato Di Tella for its hospitality. Pablo Shmerkin also thanks Boris Solomyak for useful discussions}

\begin{document}

\begin{abstract}
Under mild conditions we show that the affinity dimension of a planar self-affine set is equal to the supremum of the  Lyapunov dimensions of self-affine measures supported on self-affine proper subsets of the original set. These self-affine subsets may be chosen so as to have stronger separation properties and in such a way that the linear parts of their affinities are positive matrices. Combining this result with some recent breakthroughs in the study of self-affine measures and their associated Furstenberg measures,  we obtain new criteria under which the Hausdorff dimension of a self-affine set equals its affinity dimension. For example, applying recent results of B\'{a}r\'{a}ny, Hochman-Solomyak and Rapaport, we provide many new explicit examples of self-affine sets whose Hausdorff dimension equals its affinity dimension, and for which the linear parts do not satisfy any domination assumptions.
\end{abstract}

\maketitle

\section{Introduction and statement of main results}

Although self-affine sets and measures have been investigated since the $1980s$, it is only very recently that a comprehensive theory of their dimensions has started to emerge, especially in the planar case. In this work we apply some of the recent progress on the understanding of self-affine \emph{measures} to obtain analogous statements for self-affine \emph{sets}.

Recall that given a tuple $T=(T_1,\ldots,T_m)$ of  invertible, strictly contractive affine maps on $\R^d$, there is a unique non-empty compact set $E=E_T\subset\R^d$ such that
\[
E = \bigcup_{i=1}^m T_i(E).
\]
The set $E$ is called the \emph{self-affine set} associated to $T$. If a probability vector $p=(p_1,\ldots,p_m)$ is also given, then there exists a unique Borel probability measure $\mu=\mu_{T,p}$, supported on $E$, such that
\[
\mu = \sum_{i=1}^m p_i\, T_i\mu,
\]
where $T_i\mu(B):=\mu(T_i^{-1}B)$ for every Borel set $B\subseteq\mathbb{R}^d$. The measure $\mu$ is called the \emph{self-affine measure} associated to $(T,p)$.

The key problem on self-affine sets and measures is to determine their fractal dimensions, such as Hausdorff and box-counting dimensions in the case of sets, and, failing this, at least to determine when different notions of dimension agree. In general this problem is far from solved: even in the plane, it is not known whether or not the upper and lower box-counting dimensions of a self-affine set must always coincide. However, in 1988 Falconer \cite{Fa88} introduced a quantity associated to the linear parts $A=(A_1,\ldots,A_M)$ of the $T_i$, nowadays usually called the \emph{affinity dimension} $\adim(A)$, which is always an upper bound for the upper box-counting dimension $\overline{\dim}_B(E)$, and such that when $\|A_i\|<1/2$ for all $i$, then for almost all choices of translation tuples $(v_1,\ldots,v_M)$, the self-affine set associated to $T_i(x)=A_ix+v_i$ has both Hausdorff and box-counting dimension equal to the affinity dimension. (In fact Falconer proved this with $1/3$ as the upper bound on the norms; it was subsequently shown by Solomyak \cite{So98} that $1/2$ suffices.)

The analog of affinity dimension for measures is the Lyapunov dimension, which we denote $\ldim(\mu,A)$; see Section \ref{se:preliminaries} below for its definition. Here $\mu$ is a measure on the code space
\[
 \Sigma_M = \{1,\ldots,M\}^\mathbb{N},
\]
invariant and ergodic under the left shift $\sigma$, and the measure of interest is the projection of $\mu$ via the \emph{coding map}
\[
 \pi_T((x_i)_{i=1}^\infty) = \lim_{n\to\infty} T_{x_1}\circ\cdots\circ T_{x_n}(0).
\]
When $\mu$ is a Bernoulli measure (by a Bernoulli measure we always mean a Bernoulli measure for the canonical Markov partition of the shift space in question), its $\pi_T$-projection is a self-affine measure.

The analog of Falconer's Theorem for the Lyapunov dimension of self-affine measures was established in \cite{JoPoSi07}. It always holds that $\ldim(\mu,A)\le \adim(A)$. Conversely, A. K\"{a}enm\"{a}ki \cite{Ka04} has shown that for any tuple $A=(A_1,\ldots,A_M)$ of contractive linear maps on $\R^d$, there always exists a (not necessarily unique) ergodic measure $\mu$ on $\Sigma_m$ for which $\ldim(\mu,A)=\adim(A)$. We refer to such measures $\mu$ as \emph{K\"aenm\"aki measures}.

An important problem since Falconer's Theorem has been to provide \emph{explicit} classes of self-affine sets for which the Hausdorff dimension (or at least the box-counting dimension) agrees with the affinity dimension. Hueter and Lalley \cite{HuLa97} exhibited an open class of planar self-affine sets for which the Hausdorff and affinity dimensions agree (and are less than $1$). Falconer \cite{Fa92} and K\"{a}enm\"{a}ki and Shmerkin \cite{KaSh09} provided classes of examples for which the box-counting dimension exists and equals the affinity dimension; in these examples the dimension is larger than $1$.

A complementary strand of research concerns studying the special case in which the affine maps are diagonal and have a special row or column alignment. In this ``carpet'' case Hausdorff and box-counting dimensions may disagree with each other and with the affinity dimension, but even in the diagonal case, the expectation is that generically all dimensions should still agree. Progress in this direction has recently been obtained in  \cite{BaRaSi15}.

Very recently, a new host of techniques have been introduced by several authors which allowed dramatic progress on this circle of problems, especially in the planar case. We make a brief summary here, deferring precise definitions and statements to Section \ref{se:applications}.  B\'{a}r\'{a}ny and K\"{a}enm\"{a}ki \cite{BaKa15} (see also \cite{Ba15, FaKe15a} for earlier special cases) showed that all self-affine measures in the plane are exact-dimensional and satisfy the Ledrappier-Young formula. These results, together with classical projection theorems, give many new examples of self-affine measures for which the dimension equals the Lyapunov dimension.  Using different techniques, A. Rapaport \cite{Ra15} gave a different set of conditions that guarantee the equality of Hausdorff and Lyapunov dimensions for self-affine measures. In a different direction, M. Hochman and B. Solomyak \cite{HoSo15} calculated the dimensions of the Furstenberg measures associated to finite sets of $SL_2(\mathbb{R})$-matrices under some mild assumptions; the dimension of Furstenberg measures plays a crucial r\^{o}le in all of the recent works \cite{Ba15, BaKa15, FaKe15a, Ra15}. Also very recently, Falconer and Kempton \cite{FaKe15b} investigated the dimension of \emph{projections} of self-affine measures, and in particular gave conditions under which the dimension of the self-affine measure is preserved under all, or all but one, orthogonal projections. All of these results share the common feature that they describe the dimensions of the measures induced by Bernoulli (or at best quasi-Bernoulli) measures only, and therefore do not in principle say anything about the dimensions of self-affine sets except in certain special cases. Several of them also have an assumption of positivity or domination of the linear maps involved.

One of the main goals of this work is to show that it is always possible to approximate affinity dimension by Lyapunov dimension of Bernoulli measures, at the price of passing to an iterate of the original system and deleting some of the maps in this iterate. Moreover, if the original system is irreducible, these Bernoulli measures can be chosen so that the affine maps corresponding to their support behave in a very regular way: they strictly preserve a cone, act strongly irreducibly if this was the case for the original system, and their Lyapunov exponents and entropy approximate those of the original K\"{a}enm\"{a}ki measure (which we will demonstrate is \emph{not} a Bernoulli measure). Recall that a matrix $A\in GL_2(\R)$ is \emph{hyperbolic} if it has two real eigenvalues which are not equal in modulus. Given $\mathtt{i}=(i_1,\ldots,i_n)$ we shall write $A_{\mathtt{i}}:=A_{i_n}\cdots A_{i_1}$.
\begin{theorem} \label{th:approx-by-Bernoulli}
Let $A_i\in GL_2(\R)$, $i=1,\ldots,M$. If $\adim(A)\in (0,2)$, the $A_i$ do not preserve a proper subspace, and one of the $A_i$ is hyperbolic, then for every $\e>0$ there exist $n\in\mathbb{N}$, a set $\Gamma\subset \{1,\ldots,M\}^n$, and a Bernoulli measure $\nu$ on $\Gamma^\mathbb{N}$ such that the following hold:
\begin{enumerate}
\item $\ldim(\mu,(A_\ii)_{\ii\in\Gamma}) \ge \adim(A)-\e$. Moreover, after normalizing by dividing by $n$, the Lyapunov exponents and measure-theoretical entropy of $\nu$ are each $\e$-close to those of the K\"{a}enm\"{a}ki measure.
\item The maps $\{ A_\ii:\ii\in\Gamma\}$ strictly preserve a cone,
\item If the $A_i$ are strongly irreducible (that is, they do not preserve a finite union of proper subspaces), then so are the $(A_\ii)_{\ii\in\Gamma}$.
\end{enumerate}
Moreover, if $T_i(x)=A_i x+v_i$ are such that $(T_1,\ldots,T_M)$ satisfies the strong open set condition, then $\Gamma$ can be chosen so that additionally $(T_\ii)_{\ii\in\Gamma}$ satisfies the strong separation condition.
\end{theorem}
In particular the affinity dimension of a tuple of matrices satisfying the above conditions is thus equal to the supremum of the Lyapunov dimensions of $\sigma^n$-invariant Bernoulli measures defined on $n$-cylinders, provided that we allow these Bernoulli measures to give zero probability to certain $n$-cylinders (specifically, to cylinders which do not correspond to elements of $\Gamma$). In \S\ref{ss:insufficient-measures} below we show that the same supremum over \emph{fully-supported} Bernoulli measures can be strictly less than the affinity dimension.

Theorem \ref{th:approx-by-Bernoulli} will follow from an analysis of the K\"{a}enm\"{a}ki measure carried out in Section \ref{se:Kaenmaki-measure}. The main technical result of the paper, Theorem \ref{th:monster}, which gives more detailed information about the subsystem, and holds in a more general context, is proved in Section \ref{se:regular-subsystem}; and a separate argument to find subsystems with strong separation carried out in Section \ref{se:SOSCtoSSC}, where the proof of Theorem \ref{th:approx-by-Bernoulli} is concluded. We hope these results will find applications beyond those given in this article.

As a consequence of Theorem \ref{th:approx-by-Bernoulli}, the recent results on self-affine measures have correlates for self-affine sets. We state some of these applications here, with further examples, discussion and proofs deferred to Section \ref{se:applications}.  We say that $A=(A_1,\ldots,A_M)\in GL_2(\R)^M$ has \emph{exponential separation} if there exists a constant  $c>0$ such that if $\ii,\jj\in \{1,\ldots,M\}^n$ are distinct finite sequences, then
\[
\| A_{i_n}\cdots A_{i_1}- A_{j_n}\cdots A_{j_1}\|>c^n.
\]
We note that exponential separation implies in particular that $A_1,\ldots,A_M$ freely generate a free subgroup of $GL_2(\R)$, and when all elements of all the $A_i$ are algebraic it is equivalent to the $A_i,\ldots,A_M$ freely generating a free subgroup, see \cite{HoSo15}. 

\begin{theorem} \label{th:Lyap32}
Let $(T_1,\ldots, T_M)$ be invertible affine contractions of the plane with $T_i(x)=A_i x+v_i$, and let $E$ be the corresponding self-affine set.

Suppose that the following conditions hold:
\begin{enumerate}
\item
The transformations $A_i$ are strongly irreducible and the semigroup they generate contains a hyperbolic matrix.
\item
The affinities $(T_1,\ldots,T_M)$ satisfy the strong open set condition.
\item
The maps $A_i$ have exponential separation.
\item $\adim(A_1,\ldots,A_M) \ge \frac{3}{2}$.
\end{enumerate}
Then $\dim_H E = \adim A$.
\end{theorem}

We make some remarks on these conditions. The first assumption is very mild, and is standard in the theory of random matrix products; in this case each Bernoulli measure on $\Sigma_M$ has separated Lyapunov exponents and induces a uniquely defined Furstenberg measure. When this assumption does not hold, then $A$ has one of the following special forms (up to a change of basis):
\begin{itemize}
 \item All the $A_i$ are similarities, i.e. we are in the much better understood self-similar case.
 \item All the $A_i$ are upper triangular. This case further splits into the cases in which all the matrices are parabolic matrices or similarities (which behaves in some aspects as in the self-similar case) and the case in which at least one matrix is hyperbolic. The Hausdorff dimension of the self-affine set in this latter situation was investigated by Bara{\'n}ski \cite{Ba08}, B\'{a}r\'{a}ny \cite[Theorems 4.8 and 4.9]{Ba15}, and  B\'{a}r\'{a}ny, Rams and Simon \cite{BaRaSi16}.
 \item All the $A_i$ are either diagonal or anti-diagonal, with both cases occurring. The box-counting dimension of this class of self-affine carpets was investigated by Fraser \cite{Fr12}. We investigate their Hausdorff dimensions in Section \ref{se:irreducible-case}.
\end{itemize}

The open set condition is perhaps better known than the strong open set condition, and indeed the two are known to be equivalent in the self-similar context. However, Edgar \cite[Example 1]{Ed92} has constructed an affine iterated function system, of affinity dimension larger than $1$ and satisfying the open set condition, whose attractor is a single point. In Section \ref{se:SOSCtoSSC} we adapt Edgar's construction to show that Theorem \ref{th:Lyap32} fails if one assumes the open set condition instead of the strong open set condition. The inequivalence of the open set and strong open set conditions in the self-affine context can be seen already as a feature of Edgar's example, although to the best of our knowledge this has not previously been explicitly remarked. Our results suggest to us that for affine iterated function systems it is the strong open set condition and not the open set condition which is the most natural and appropriate separation hypothesis.

The exponential separation condition arises from the work of Hochman and Solomyak \cite{HoSo15}. It is plausible that it is a generic condition among tuples of matrices in $SL_2(\R)$, but this is not currently known. On the other hand, we note that if this condition holds for  $(A_1,\ldots,A_M)$, then it also holds for $(r_1 A_1,\ldots,r_M A_M)$ for any scalars $r_i\neq 0$ (see the proof of Corollary \ref{co:Hochman-Solomyak}). Also, when the matrices $A_i$ have algebraic coefficients, it holds if and only if the $A_i$ freely generate a subgroup of $SL_2(\R)$, see \cite[Lemma 6.1]{HoSo15}. Unfortunately, the freeness of matrix semigroups is in general very difficult to check: for three-dimensional non-negative integer matrices, the problem of determining freeness is known to be computationally undecidable \cite{KlBiSa91}. A particularly vivid example of the difficulty of the two-dimensional problem may be found in \cite{CaHaKa99,GaGuKi10}. Nevertheless, one can construct many examples of free semigroups of $SL_2(\R)$ with algebraic coefficients: see  \S\ref{ss:examples} below.

In the final condition, the value $3/2$ is likely an artifact of the proof. Affinity dimension is in general difficult to compute, but the condition can still be easily checked in many cases. For example, it is satisfied if
\[
\sum_{i=1}^M |\det A_i|^{\frac{3}{4}}\ge 1 .
\]

We remark that when the affinity dimension equals $2$ and the open set condition holds, then the self-affine set automatically has positive Lebesgue measure, while the open set condition cannot hold if the affinity dimension exceeds $2$. See Lemma \ref{le:affin-ge-2-OSC} for these standard facts.

The next application weakens the analogous conditions given by Hueter and Lalley \cite{HuLa97} and B\'{a}r\'{a}ny \cite{Ba15} for the equality of Hausdorff and affinity dimension. In particular, we do not require domination.
\begin{theorem} \label{th:HueterLalley}
Let $(T_1,\ldots, T_M)$ be invertible affine contractions, with $T_i(x)=A_i x+v_i$, and let $E$ be the corresponding self-affine set.

Suppose that the following conditions hold:
\begin{enumerate}
\item
The transformations $A_i$ are strongly irreducible and the semigroup they generate contains a hyperbolic matrix.
\item
The affinities $(T_1,\ldots,T_M)$ satisfy the strong open set condition.
\item
The maps $A_i$ have exponential separation.
\item The matrices $A_i$ satisfy the bunching condition $\alpha_1(A_i)^2 \le \alpha_2(A_i)$ for all $i$.
\end{enumerate}
Then $\dim_H E = \adim A$.
\end{theorem}

Note that the first three conditions are the same as in Theorem \ref{th:Lyap32}. The r\^{o}le of the bunching condition (together with the other assumptions) is to ensure that either the dimension of the Furstenberg measure is $1$, or it is larger than the affinity dimension. This allows the application of Theorem \ref{th:Barany}. We remark that the separation hypothesis of Hueter and Lalley in \cite{HuLa97} can be easily seen to imply condition (3) above, since under that hypothesis the images of the negative diagonal line in $\mathbb{R}^2$ under two distinct products $A_{i_n}\cdots A_{i_1}$, $A_{j_n}\cdots A_{j_1}$ must be exponentially separated.

We conclude this introduction by putting our results in a wider context. According to a folklore conjecture in the field, equality of Hausdorff and affinity dimensions should occur for an open and dense family of affine iterated function systems, at least under suitable separation assumptions. Several of the results described above support an even stronger version of the conjecture: for an open and dense set of tuples $(A_1,\ldots,A_M)$ of strictly contractive linear bijections of $\mathbb{R}^2$, and for \emph{every} choice of translations $v_1,\ldots,v_M$ such that $T_i(x)=A_i x+v_i$ satisfies the strong open set condition, the Hausdorff dimension of the invariant set equals the affinity dimension. (We speculate that this may even be true whenever $A_1,\ldots,A_M$ generate a Zariski dense subgroup of $GL_2(\mathbb{R})$.) Our results provide additional evidence for this conjecture by showing for the first time that there are tuples $(A_1,\ldots,A_M)$ verifying the conjecture which do not satisfy domination (we recall that lack of domination holds in non-empty open subsets of parameter space). Moreover, it follows from Theorems \ref{th:Lyap32} and \ref{th:HueterLalley} that such tuples $(A_1,\ldots,A_M)$ are in fact dense in large open subsets of parameter space: firstly, in the set of all tuples of affinity dimension strictly greater than $3/2$ (which is open since affinity dimension is continuous, see \cite{FeSh14}); and secondly, in the set of all tuples satisfying the bunching condition of Theorem \ref{th:HueterLalley}. We direct the reader to \S\ref{ss:examples} below for some additional discussion including concrete examples.

\section{Preliminaries}
\label{se:preliminaries}

In this section we review some of the main concepts and results in the theory of self-affine sets, and set up notation along the way. We restrict ourselves to the planar case, and refer to \cite{Ka04} for details and proofs. We recall that $GL_2(\mathbb{R})$, $GL_2^+(\mathbb{R})$ and $SL_2(\mathbb{R})$  denote the sets of $2\times 2$ real matrices whose determinant is respectively nonzero, positive, or equal to $1$. A set or tuple of elements of $GL_2(\mathbb{R})$ will be called \emph{irreducible} if its members do not preserve a common invariant one-dimensional subspace, and \emph{strongly irreducible} if they do not commonly preserve a finite union of one-dimensional subspaces. Throughout this article $\|\cdot\|$ denotes the Euclidean metric on $\mathbb{R}^2$ or the operator norm on $GL_2(\mathbb{R})$ derived therefrom, the distinction between the two being obvious from context.

Given a matrix $A\in GL_2(\R)$, its \emph{singular values} $\alpha_1(A)\ge \alpha_2(A)$ are the positive square roots of the eigenvalues of the positive definite matrix $A^*A$. In particular $|\det A|\equiv \alpha_1(A)\alpha_2(A)$, $\alpha_1(A)\equiv \|A\|$ and $\alpha_2(A)\equiv \|A^{-1}\|^{-1}$.

For $s\ge 0$, the \emph{singular value function} (SVF) $\varphi^s:GL_2(\R)\to\R$ is defined as
\[
\varphi^s(A) = \left\{
\begin{array}{ll}
  \alpha_1(A)^s & \text{ if } 0\le s < 1 \\
  \alpha_1(A) \alpha_2(A)^{s-1} & \text{ if } 1\le s <2 \\
  |\det(A)|^{s/2} & \text{ if } 2\le s
\end{array}
 \right..
\]
The singular value function is well-known to satisfy the submultiplicativity property $\varphi^s(AB)\le \varphi^s(A)\varphi^s(B)$ for every $A,B\in GL_2(\mathbb{R})$. Given a tuple $A=(A_1,\ldots,A_M)\in GL_2(\R)^M$, the associated topological pressure is defined as
\[
P(\varphi^s,A) = \lim_{n\to\infty} \frac{1}{n} \log\left( \sum_{\ii\in \{1,\ldots,M\}^n} \varphi^s(A_{\ii_1}\cdots A_{\ii_n}) \right),
\]
where the limit exists by sub-multiplicativity of $\varphi^s$. The pressure function $s\mapsto P(\varphi^s,A)$ is convex and continuous. If additionally every $A_i$ has norm strictly less than one, then $s\mapsto P(\varphi^s,A)$ is strictly decreasing and there exists a unique $s\ge 0$ for which $P(\varphi^s,A)=0$: in this case the \emph{affinity dimension} $\adim(A)$ of $A$ is defined to be this unique number $s$.

Given $\mathtt{i}=(i_1,\ldots,i_n)\in\{1,\ldots,M\}^n$ and $\mathtt{j}=(j_1,\ldots,j_m)\in\{1,\ldots,M\}^m$ we let $\mathtt{ij}$ denote their concatenation $(i_1,\ldots,i_n,j_1,\ldots,j_m)$. Given $\mathtt{i}=(i_1,\ldots,i_n)$ and $A=(A_1,\ldots,A_M)$  we will also find it convenient to write $|\mathtt{i}|=n$ and $A_{\mathtt{i}}:=A_{i_n}\cdots A_{i_1}$.

Given $A=(A_1,\ldots,A_M)\in GL_2(\mathbb{R})$, we define
\[A(x,n):=A_{x_n}\cdots A_{x_1}\]
for every $x \in \Sigma_M$ and $n \geq 1$, noting that this definition implies the  \emph{cocycle identity} $A(x,n_1+n_2)=A(\sigma^{n_1}x,n_2)A(x,n_1)$ for every $x \in \Sigma_M$ and $n_1,n_2\geq 1$.  For every $\sigma$-invariant measure $\mu$ on $\Sigma_M$ we define the \emph{Lyapunov exponents} of $A$ with respect to $\mu$ to be the quantities
\[\lambda_1(\mu):=\lim_{n\to\infty}\frac{1}{n}\int \log\alpha_1(A(x,n))\,d\mu(x)=\inf_{n\geq 1}\frac{1}{n}\int \log\alpha_1(A(x,n))\,d\mu(x), \]
\[\lambda_2(\mu):=\lim_{n\to\infty}\frac{1}{n}\int \log\alpha_2(A(x,n))\,d\mu(x)=\sup_{n\geq 1}\frac{1}{n}\int \log\alpha_1(A(x,n))\,d\mu(x), \]
where the limit defining $\lambda_1(\mu)$ (resp. $\lambda_2(\mu)$) exists by subadditivity (resp. superadditivity).  Combining these definitions with that of $\varphi^s$ it follows easily that
\[\lim_{n\to\infty}\frac{1}{n}\int \log\varphi^s(A(x,n))\,d\mu(x)=
\left\{
\begin{array}{ll}
  s\lambda_1(\mu) & \text{ if } 0\le s \le 1 \\
  \lambda_1(\mu)+ (s-1)\lambda_2(\mu)& \text{ if } 1\le s \le 2 \\
  \frac{s}{2}\left(\lambda_1(\mu)+\lambda_2(\mu)\right) & \text{ if } 2\le s
\end{array}
 \right..
\]
By the subadditive variational principle (see \cite{CaFeHu08}) we have
\[
P(\varphi^s,A) =\sup_\mu \left( h(\mu) + \inf_{n \geq 1} \frac{1}{n}\int \log\varphi^s(A(x,n)) \,d\mu(x)\right)
\]
where the supremum is taken over all $\sigma$-invariant Borel probability measures $\mu$ on $\Sigma_M$, and $h(\mu)$ denotes metric entropy. Measures
which attain this supremum are called \emph{equilibrium states} for $\varphi^s$, and for every $A$ and $s$ at least one equilibrium state exists. In the case $s=\adim(A)$ we also call these equilibrium states \emph{K\"{a}enm\"{a}ki measures}. The \emph{Lyapunov dimension} $\ldim$ of $\mu$ is defined as
\[
\ldim(\mu,A)  = \left\{
\begin{array}{ll}
  \frac{h(\mu)}{-\lambda_1(\mu)} & \text{ if } h(\mu) < -\lambda_1(\mu) \\
  1 + \frac{h(\mu)+\lambda_1(\mu)}{-\lambda_2(\mu)} & \text{ if } -\lambda_1(\mu) \le h(\mu) < -\lambda_2(\mu) \\
  2\frac{h(\mu)}{-\lambda_1(\mu)-\lambda_2(\mu)} & \text{ if } -\lambda_2(\mu) \le h(\mu) \end{array}
 \right..
\]
Then $\ldim(\mu,A) \le \adim(A)$, with equality if and only if $\mu$ is  K\"{a}enm\"{a}ki measure. We sometimes write $\ldim(\mu)$ instead of $\ldim(\mu,A)$ when the tuple $A$ is clear from context.

\section{Properties of equilibrium states for the Singular Value Function}
\label{se:Kaenmaki-measure}
\subsection{Principal results}
Perhaps surprisingly, the existing literature contains relatively few facts about the ergodic properties of equilibrium states for $\varphi^s$. In this section we prove the following theorem on the equilibrium states of $\varphi^s$ in two dimensions:
\begin{theorem} \label{th:properties-of-mu}
Let $A_1,\ldots,A_M \in GL_2(\mathbb{R})$. Suppose that the matrices $A_1,\ldots,A_M$ do not have a common one-dimensional invariant subspace, and that at least one of them is hyperbolic. Let $0<s<2$, and let $\mu$ be a Borel probability measure on $\Sigma_M$ which is an equilibrium state for $\varphi^s$. Then $\mu$ is 
globally supported on $\Sigma_M$, and the Lyapunov exponents $\lambda_1(\mu)$ and $\lambda_2(\mu)$ are unequal.
\end{theorem}

At several points the proof of Theorem \ref{th:properties-of-mu} splits depending on whether or not $(A_1,\ldots,A_M)$ is strongly irreducible. In the next lemma we characterize the structure of the $A_i$ in the irreducible but not strongly irreducible situation. This characterization is certainly well-known, but we include the proof for the reader's convenience. We recall that a matrix is called \emph{anti-diagonal} if all elements off the top-right to lower-left diagonal are zero.
\begin{lemma} \label{le:irr-not-strong-irr}
Let $A_1,\ldots,A_M \in GL_2(\mathbb{R})$. Suppose that the matrices $A_1,\ldots,A_M$ do not have a common one-dimensional invariant subspace, that one of them is hyperbolic, and that there is a finite union of one-dimensional subspaces which is invariant under all $A_i$. Then after a change of basis all the $A_i$ are either diagonal or anti-diagonal, with both cases occurring.
\end{lemma}
\begin{proof}
After a change of basis, we can assume the given hyperbolic matrix $A_j$ is diagonal. The projective orbit of any non-principal line under $A_j$ is infinite, so the only non-trivial set of lines that is fixed by all the $A_i$ is $\{ e_0, e_1\}$, the standard basis of $\R^2$. This means that all the $A_i$ either map $e_i$ to $e_i$ (in which case they are diagonal), or $e_i$ to $e_{1-i}$ (in which case they are anti-diagonal), and the latter case must occur since $e_0$ is not invariant under all $A_i$.
\end{proof}

For the proof of Theorem  \ref{th:properties-of-mu} we will rely on the following Gibbs property of the K\"{a}enm\"{a}ki measure in the irreducible case; see \cite[Propositions 2.3 and 3.4 and Theorem 3.7]{KaRe14}:
\begin{proposition} \label{prop:Kaenmaki}
Let $A=(A_1,\ldots,A_M) \in GL_2(\mathbb{R})$ be irreducible. Then there exists a unique equilibrium state $\mu$ for $\varphi^s$. This measure is ergodic and satisfies the following Gibbs property: there exists $C>0$ such that
\begin{equation} \label{eq:Gibbs-property}
C^{-1}  \le \frac{\mu([\ii])}{\varphi^s(A_\ii) e^{n P(\varphi^s,A)}} \le C
\end{equation}
for all finite words $\ii\in\{1,\ldots,M\}^n$.
\end{proposition}

Note that the fact that $\mu$ is globally supported follows at once from this proposition.
Let us show that $\mu$ has simple Lyapunov exponents, that is, that $\lambda_1(\mu)\neq \lambda_2(\mu)$. For this, we rely on:
\begin{lemma}
Suppose that $\mu$ is an equilibrium state for $\varphi^s$ such that $\lambda_1(\mu)=\lambda_2(\mu)$. Then $\mu$ is an equilibrium state for the function $A\mapsto |\det A|^{s/2}$; that is, it maximizes the expression
\[h(\nu)+\int \log|\det A_{x_1}| d\nu(x)\]
over all $\sigma$-invariant Borel probability measures $\nu$ on $\Sigma_M$. In particular, $\mu$ is a Bernoulli measure.
\end{lemma}
\begin{proof}
Since $\mu$ has equal Lyapunov exponents
\[\inf_{n \geq 1}\frac{1}{n}\int \log \varphi^s\left(A(x,n)\right)d\mu(x)=\frac{s}{2}\left(\lambda_1(\mu)+\lambda_2(\mu)\right)=\frac{s}{2}\int \log|\det A_{x_1}|d\mu(x).\]
If the conclusion of the lemma is false then there exists a measure $\nu$ such that
\[h(\nu)+\int_{\Sigma_M} \log|\det A_{x_1}|d\nu(x)>h(\mu)+\int_{\Sigma_M} \log|\det A_{x_1}|d\mu(x),\]
but then we have
\begin{align*}
h(\nu)+\inf_{n\geq 1}\frac{1}{n}\int \log \varphi^s(A(x,n))d\nu(x)
&\geq h(\nu)+\frac{s}{2}\int \log|\det A_{x_1}| d\nu(x)\\
&> h(\mu)+\frac{s}{2}\int \log|\det A_{x_1}| d\mu(x)\\
&=h(\mu)+\inf_{n \geq 1}\frac{1}{n}\int \log \varphi^s\left(A(x,n)\right)d\mu(x),
\end{align*}
using the elementary inequality $\varphi^s(A)\geq |\det A|^{s/2}$ together with the invariance of $\nu$. In particular $\mu$ is not an equilibrium state for $\varphi^s$, which is a contradiction. The fact that $\mu$ is a Bernoulli measure follows from the fact that $\log\left(|\det A_{x_1}|^{s/2}\right)$ depends only on the first co-ordinate of $x\in \Sigma_M$.\end{proof}

To conclude the proof of  Theorem  \ref{th:properties-of-mu}, we again distinguish two cases: the case in which the system is strongly irreducible, and that in which it is irreducible but not strongly irreducible. In the first case, we know from Furstenberg's Theorem (see e.g. \cite[p.30]{BoLa85}) that Lyapunov exponents for Bernoulli measures are distinct, so we obtain a contradiction with the previous lemma. From now on we assume we are in the latter case. In light of Lemma  \ref{le:irr-not-strong-irr} and the previous lemma, the proof of Theorem \ref{th:properties-of-mu} will be finished once we establish the following.
\begin{lemma}\label{le:Bernoulli-not-Gibbs}
Let $A_1,\ldots,A_M \in GL_2(\mathbb{R})$. Suppose that at least one matrix $A_i$ is diagonal and hyperbolic, and that at least one other matrix is anti-diagonal. Then for every $0<s<2$, the equilibrium state of $(A_1,\ldots,A_M)$ for $\varphi^s$ is not a Bernoulli measure.
\end{lemma}
\begin{proof}
The system is irreducible thanks to the presence of the anti-diagonal matrix. We can then apply Proposition \ref{prop:Kaenmaki}. Suppose that $\mu$ is an equilibrium state of $(A_1,\ldots,A_M)$ for $\varphi^s$ which is also a Bernoulli measure, and let $\ii,\jj\in\{1,\ldots,M\}^n$ be permutations of each other. Since $\mu$ is a Bernoulli measure, $\mu([\ii])=\mu([\jj])$. Hence, the Gibbs property \eqref{eq:Gibbs-property} implies the inequality
\[
\varphi^s(A_\ii) \le C^2 \varphi^s(A_\jj)
\]
independently of $n$. Now suppose that $A_i$ is hyperbolic and diagonal and that $A_j$ is anti-diagonal. It is easy to check that
\[
\frac{\varphi^s(A_i^{2n} A_j)}{\varphi^s(A_i^n A_j A_i^n)} \to \infty
\]
as $n\to\infty$, and this contradiction finishes the proof.
\end{proof}

\subsection{Insufficiency of fully-supported Bernoulli measures}\label{ss:insufficient-measures}
The results in this section suffice to prove the assertion made below the statement of Theorem \ref{th:approx-by-Bernoulli}: there exists a tuple $A=(A_1,\ldots,A_M)$ such that $\adim(A)$ is \emph{not} equal to the supremum of $\ldim(A,\mu)$ taken over all \emph{fully-supported} probability measures $\mu$ which are $\sigma^n$-invariant Bernoulli measures for some integer $n\geq 1$. To see this let $A=(A_1,\ldots,A_M)$ be given by a mixture of anti-diagonal matrices and diagonal matrices, with at least one matrix being hyperbolic.
To simplify the argument we shall assume additionally that $0<\adim(A)\leq 1$, but the case in which $1<\adim(A)<2$ may be handled similarly. Let $s:=\adim(A)\leq 1$ and consider the two pressures
\[P_1(A,s):=\lim_{n\to\infty} \frac{1}{n}\log \sum_{|\mathtt{i}|=n}\left|\det A_{\mathtt{i}}\right|^{\frac{s}{2}}=\sup_\mu \left[h(\mu)+\int \log|\det A_{x_1}|d\mu(x)\right], \]
\[P_2(A,s):=\lim_{n\to\infty} \frac{1}{n}\log \sum_{|\mathtt{i}|=n}\varphi^s(A_{\mathtt{i}})=\sup_\mu \left[h(\mu)+s\lambda_1(\mu)\right].\]
Since we always have $|\det A_{\mathtt{i}}|^{s/2}\leq \varphi^s(A_{\mathtt{i}})$ it follows that $P_1(A,s)\leq P_2(A,s)$.  If the two pressures are equal then by the same inequality any equilibrium state for $P_1$ must be an equilibrium state for $P_2$, but such an equilibrium state must be a Bernoulli measure since the potential $\log |\det A_{x_1}|$ depends only on the first co-ordinate of $x \in \Sigma_M$. By Lemma \ref{le:Bernoulli-not-Gibbs} this is impossible, and therefore $P_1(A,s)<P_2(A,s)$.

Now suppose that $\nu$ is a Bernoulli measure for $\sigma^n$ with full support. In this case one may show that the particular structure of the matrices $A_1,\ldots,A_M$ implies that the Lyapunov exponents $\lambda_1(\nu)$, $\lambda_2(\nu)$ must be equal (see e.g. \cite[p.38]{BoLa85}). Applying the variational principle for the transformation $\sigma^n$ it follows that
\[h(\nu)+s\lambda_1(\nu)=h(\nu)+\frac{s}{2}\left(\lambda_1(\nu)+\lambda_2(\nu)\right) \leq nP_1(A,s)<nP_2(A,s)=0\]
and since $s=\adim(A)\leq 1$ we have $h(\nu)=-s\lambda_1(\nu)\leq -\lambda_1(\nu)$ and therefore
\begin{align*}
\ldim(A,\nu) &= \frac{h(\nu)}{-\lambda_1(\nu)} \leq s+ \frac{nP_1(A,s)}{-\lambda_1(\nu)} \leq s +\frac{P_1(A,s)}{-\min_{1\leq i \leq M}\frac{1}{2}\log |\det A_i|} \\
&=\adim(A) +\frac{P_1(A,s)}{-\min_{1\leq i \leq M}\frac{1}{2}\log |\det A_i|}
\end{align*}
which is less than $\adim(A)$ by an amount not depending on $\nu$ or $n$. This completes the proof of the assertion.

\section{Regular subsystems}
\label{se:regular-subsystem}
We recall some further definitions.  Given a set $\mathsf{A}$ of matrices in $\R^{2\times 2}$, its \emph{joint spectral radius} and \emph{lower spectral radius} are given, respectively, by
\begin{align*}
\inf_{n \geq 1} & \sup_{B_1,\ldots,B_n \in \mathsf{A}} \|B_1\cdots B_n\|^{1/n} ,\\
\inf_{n \geq 1} & \inf_{B_1,\ldots,B_n \in \mathsf{A}} \|B_1\cdots B_n\|^{1/n} .\end{align*}
In both cases the infimum is also a limit: see for example \cite{Ju09}. It follows easily that both quantities are independent of the choice of norm and/or basis on $\mathbb{R}^2$.

We let $\mathbb{RP}^1$ denote the real projective line, which is the set of all lines through the origin in $\mathbb{R}^2$. We let $\overline{u}\in\mathbb{RP}^1$ denote the line generated by the nonzero vector $u \in\mathbb{R}^2$. We equip $\mathbb{RP}^1$ with the metric $d$ given by
\[d(\overline{u},\overline{v})=\frac{\|u\wedge v\|}{\|u\|\cdot\|v\|}\]
for nonzero $u \in \overline{u}$, $v \in \overline{v}$. Clearly the choice of $u \in \overline{u}$, $v \in \overline{v}$ in the definition is immaterial when $\overline{u}$ and $\overline{v}$ are fixed. Since
\[\|u\wedge v\|^2=\langle u,u\rangle \langle v,v\rangle - \langle u,v\rangle^2 = \|u\|^2\|v\|^2(1-\cos^2\angle (u,v))\]
this metric defines the distance between two subspaces to be the sine of the angle between them. We will abuse notation by writing $A$ to denote the projective linear transformation $\mathbb{RP}^1\to\mathbb{RP}^1$ induced by an invertible matrix $A \in\mathbb{R}^{2\times 2}$ as well as the matrix itself.

For the purposes of this article a \emph{cone} in $\mathbb{R}^2$ is a closed, positively homogenous, convex subset of $\mathbb{R}^2\setminus \{0\}$ with nonempty interior. We say that a matrix $A$ \emph{strictly preserves} a cone $\mathcal{C}$ if $A\mathcal{C}$ is a subset of the interior of $\mathcal{C}$, and we say that a (finite) set of matrices strictly preserves $\mathcal{C}$ if this is true of all of its elements. We note that a set $\mathcal{C}$ is a cone if and only if there exists a closed projective interval $\mathcal{K}\subset \mathbb{RP}^1$ such that $\mathcal{C}$ is one of the two connected components of the set $\{u \in\mathbb{R}^2\setminus \{0\}\colon\overline{u}\in\mathcal{K}\}$. In view of this it is easy to see that a matrix $A$ (strictly) preserves a cone in $\mathbb{R}^2$ if and only if there exists a basis in which its entries are all (strictly) positive.

We recall the following version of Oseledets' multiplicative ergodic theorem in the plane:
\begin{theorem} \label{th:Oseledets}
Let $\sigma$ be an invertible measure-preserving transformation of the probability space $(X,\mathcal{F},\mu)$ and let $A \colon X \times \mathbb{Z} \to GL_2(\mathbb{R})$ be a measurable linear cocycle such that $\int \left|\log \|A(x,1)\|\right|d\mu(x)<\infty$. Define
\[\lambda_i:=\lim_{n\to\infty} \frac{1}{n}\int \log\alpha_i(A(x,n))d\mu(x)\]
for $i=1,2$, and suppose that these two values are unequal. Then there exist measurable functions $\mathfrak{u},\mathfrak{s} \colon X \to \mathbb{RP}^1$ such that for $\mu$-a.e. $x \in X$
\begin{enumerate}[(i)]
\item
$A(x,n)\mathfrak{u}(x)=\mathfrak{u}(\sigma^nx)$ and $A(x,n)\mathfrak{s}(x)=\mathfrak{s}(\sigma^nx)$
\item
For all nonzero $u \in\mathfrak{u}(x)$ and $v \in \mathfrak{s}(x)$,
\[\lim_{n \to \infty} \frac{1}{n}\log \|A(x,n)u\| = \lambda_1,\]
\[\lim_{n \to \infty} \frac{1}{n}\log \|A(x,n)v\| = \lambda_2.\]
\end{enumerate}
\end{theorem}

The technical core of Theorem \ref{th:approx-by-Bernoulli} is the following general result, which is rooted in ideas of \cite{FeSh14}.
\begin{theorem}\label{th:monster}
Let $A_1,\ldots,A_M \in GL_2(\mathbb{R})$, let $\mu$ be a fully-supported ergodic invariant measure on $\Sigma_M$, and let $n_0\geq 1$ and $\varepsilon>0$. Suppose that the Lyapunov exponents $\lambda_1(\mu)$, $\lambda_2(\mu)$ defined by
\[\lambda_i(\mu):=\lim_{n\to\infty} \frac{1}{n}\log \int \alpha_i(A(x,n))d\mu(x)\]
are not equal to one another. Then there exist $n>n_0$ and a subset $\Gamma$ of $\{1,\ldots,M\}^n$ such that:
\begin{enumerate}[(i)]
\item\label{it:tm1}
The cardinality of $\Gamma$ is at least $e^{n(h(\mu)-\varepsilon)}$.
\item\label{it:tm2}
The matrices $\{A_{\mathtt{i}}\colon \mathtt{i}\in \Gamma\}$ strictly preserve a cone $\mathcal{C}$.
\item\label{it:tm2b}
For every $u \in \mathcal{C}$ and $\mathtt{i}\in\Gamma$ we have $\|A_{\mathtt{i}}u\|\geq e^{n(\lambda_1(\mu)-\varepsilon)}\|u\|$. In particular, the set of matrices $\{A_{\mathtt{i}}\colon \mathtt{i}\in \Gamma\}$ has lower spectral radius at least $e^{n(\lambda_1(\mu)-\varepsilon)}$.
\item\label{it:tm3}
The set of matrices $\{A_{\mathtt{i}}\colon \mathtt{i}\in \Gamma\}$ has joint spectral radius at most $e^{n(\lambda_1(\mu)+\varepsilon)}$.
\item\label{it:tm4}
For every $\mathtt{i}\in \Gamma$ we have $e^{n(\lambda_1(\mu)+\lambda_2(\mu)-\varepsilon)} \leq \det A_{\mathtt{i}} \leq e^{n(\lambda_1(\mu)+\lambda_2(\mu)+\varepsilon)}$. In particular $\{A_{\mathtt{i}}\colon \mathtt{i}\in \Gamma\}\subset GL_2^+(\mathbb{R})$.
\item\label{it:tm5}
If $\{A_1,\ldots,A_M\}$ is strongly irreducible then so is $\{A_{\mathtt{i}} \colon \mathtt{i}\in \Gamma\}$.\item
If $\mathtt{k}\in \{1,\ldots,M\}^{k}$ where $1 \leq k \leq n_0$, then $\mathtt{k}$ is a subword of every $\mathtt{i}\in\Gamma$.
\end{enumerate}
\end{theorem}
To see that the growth inequality for vectors $u \in \mathcal{C}$ implies that $\{A_{\mathtt{i}}\colon \mathtt{i}\in\Gamma\}$ has lower spectral radius at least $e^{n(\lambda_1(\mu)-\varepsilon)}$, we note that if $\mathtt{i}_1,\ldots,\mathtt{i}_m \in \Gamma$ and $u \in \mathcal{C}$ is a unit vector then
\[\left\|A_{\mathtt{i}_m}\cdots A_{\mathtt{i}_1}\right\| \geq \left\|A_{\mathtt{i}_m}\cdots A_{\mathtt{i}_1}u \right\| \geq e^{nm(\lambda_1(\mu)-\varepsilon)}\]
since each of these matrices maps $\mathcal{C}$ back into itself and the lower estimate \eqref{it:tm2b} can thus be applied $m$ times iteratively.

We remark that $n$ may be taken arbitrarily large if so desired, since if $\Gamma$ has the properties described above then so does the set $\Gamma':=\{\mathtt{i}_1\cdots \mathtt{i}_k \colon \mathtt{i}_j\in\Gamma\}$ for every integer $k \geq 1$. We will begin by proving a reduced version of Theorem \ref{th:monster}, and then extend the reduced version to the full statement using two subsequent lemmas. The reduced form of Theorem \ref{th:monster} is:
\begin{proposition}  \label{pr:submonster}

Let $A_1,\ldots,A_M \in GL_2(\mathbb{R})$, let $\mu$ be a fully-supported ergodic invariant measure on $\Sigma_M$, and let $n_0\geq 1$ and $\varepsilon>0$. Suppose that the Lyapunov exponents of $\mu$ are unequal. Then there exist $n>n_0$ and a subset $\Gamma$ of $\{1,\ldots,M\}^n$ such that:
\begin{enumerate}[(i)]
\item\label{it:pm1}
The set $\Gamma$ has cardinality strictly greater than $e^{n(h(\mu)-\varepsilon)}$.
\item\label{it:pm2}
There exists a closed projective interval $\mathcal{K}\subset \mathbb{RP}^1$ such that for every $\mathtt{i} \in \Gamma$ the set $A_{\mathtt{i}}\mathcal{K}$ is contained in the interior of $\mathcal{K}$.
\item\label{it:pm3}
The set of matrices $\{A_{\mathtt{i}}\colon \mathtt{i}\in \Gamma\}$ has joint spectral radius at most $e^{n(\lambda_1(\mu)+\varepsilon)}$.
\item\label{it:pm4}
If $u \in \overline{u}\in\mathcal{K}$ and $\mathtt{i}\in \Gamma$ we have $\|A_{\mathtt{i}}u\|\geq e^{n(\lambda_1(\mu)-\varepsilon)}\|u\|$.
\item\label{it:pm5}
For every $\mathtt{i}\in \Gamma$ we have $e^{n(\lambda_1(\mu)+\lambda_2(\mu)-\varepsilon)} \leq |\det A_{\mathtt{i}}| \leq e^{n(\lambda_1(\mu)+\lambda_2(\mu)+\varepsilon)}$.
\item\label{it:pm6}
If $\mathtt{k}\in \{1,\ldots,M\}^{k}$ where $1 \leq k \leq n_0$, then $\mathtt{k}$ is a subword of every $\mathtt{i}\in\Gamma$.
\end{enumerate}
\end{proposition}
\begin{proof}
We assume without loss of generality that
\begin{equation} \label{eq:monster0}
\varepsilon<(\lambda_1(\mu)-\lambda_2(\mu))/4.
\end{equation}
In order to apply the multiplicative ergodic theorem we require invertibility of the underlying measure-preserving transformation, so by abuse of notation we replace, for the remainder of the proof, the ergodic measure-preserving system $(\sigma,\Sigma_M,\mu)$ with its invertible natural extension.

Let $\nu$ denote the measure on $\mathbb{RP}^1 \times \mathbb{RP}^1$ given by
\[\nu(B):=\mu\left(\{x \in \Sigma_M \colon (\mathfrak{u}(x),\mathfrak{s}(x))\in B\}\right)\]
and observe that $\nu$ gives zero measure to the diagonal of $\mathbb{RP}^1\times \mathbb{RP}^1$. Let $(\overline{w_u},\overline{w_s})$ be in the support of $\nu$ with $\overline{w_u} \neq \overline{w_s}$, and choose $\delta>0$ such that
\[\left\{\overline{u} \in \mathbb{RP}^1 \colon d(\overline{u},\overline{w_u})\leq 2\delta\right\} \cap \left\{\overline{u} \in \mathbb{RP}^1 \colon d(\overline{v},\overline{w_s})\leq \delta \right\} =\emptyset.\]
Let $Z:=\{x \in \Sigma_M \colon d(\mathfrak{u}(x),\overline{w_u})\leq \delta \text{ and }d(\mathfrak{s}(x),\overline{w_s}) \leq \delta\}$. Define $\mathcal{K}:=\{\overline{u} \in \mathbb{RP}^1 \colon d(\overline{u},\overline{w_u})\leq 2\delta\}$ and $\mathcal{J}:=\{\overline{u} \in \mathbb{RP}^1 \colon d(\overline{u},\overline{w_s})\leq \delta\}$. Choose $\tau>0$ such that $\|u \wedge v\| \geq \tau$ whenever $u$ and $v$ are unit vectors with $\overline{u} \in \mathcal{K}$ and $\overline{v} \in \mathcal{J}$.

We know that for almost every $x \in \Sigma_M$,
\begin{equation}\label{eq:monsterA1}
\lim_{n \to \infty} \frac{1}{n}\log \left(\frac{\|A(x,n)u\|}{\|u\|}\right) =\lambda_1
\end{equation}
uniformly over nonzero vectors $u\in\mathfrak{u}(x)$, and
\begin{equation}\label{eq:monsterA2}
\lim_{n \to \infty} \frac{1}{n}\log \left(\frac{\|A(x,n)v\|}{\|v\|}\right)=\lambda_2
\end{equation}
uniformly over nonzero vectors $v\in\mathfrak{s}(x)$, by the multiplicative ergodic theorem; and by the Birkhoff ergodic theorem,
\begin{equation}\label{eq:monsterA3}
\lim_{n \to \infty}\frac{1}{n}\log |\det A(x,n)|=\lambda_1+\lambda_2
\end{equation}
for almost every $x \in \Sigma_M$. Since $\mu$ is fully-supported, we have for every word $\mathtt{k}$ of length at most $n_0$
\begin{equation}\label{eq:monsterA4}
\lim_{n\to\infty} \frac{1}{n} \sum_{i=0}^{n-1}\mathbf{1}_{[\mathtt{k}]}(\sigma^ix) = \mu([\mathtt{k}])>0
\end{equation}
for almost every $x \in \Sigma_M$,
and by the Shannon-McMillan-Breiman theorem
\begin{equation}\label{eq:monsterA5}
\lim_{n \to \infty} \frac{1}{n}\log \mu\left(\left[x_1\cdots x_n\right]\right)=-h(\mu)
\end{equation}
for $\mu$-a.e. $x\in \Sigma_M$. Lastly, by the subadditive ergodic theorem we have for $\mu$-a.e. $x$
\begin{equation}\label{eq:monsterA6}
\lim_{n\to\infty} \frac{1}{n}\log\|A(x,n)\| = \lambda_1(\mu).
\end{equation}

We now construct a subset of $X$ on which the above properties hold uniformly, within suitable tolerances, for a particular time $n$.
Let $\kappa:=\frac{1}{3}\mu(Z)^2$.  Since equations \eqref{eq:monsterA1}--\eqref{eq:monsterA6} converge pointwise, in particular they converge in measure. It follows that for every sufficiently large $n$ the following statements hold for all $x$ belonging to a set $Y_n$ such that $\mu(Y_n)>1-\kappa$:
\begin{equation}\label{eq:monsterB1}
 \log  \|A(x,n)u\| \geq n(\lambda_1(\mu)-\frac{\varepsilon}{2})-\log \tau
\end{equation}
and
\begin{equation}\label{eq:monsterB2}
\log \|A(x,n)v\| \leq n(\lambda_2(\mu)+\frac{\varepsilon}{2})+\log \tau
\end{equation}
for every unit vector $u \in\mathfrak{u}(x)$ and every unit vector $v \in \mathfrak{s}(x)$; and also
\begin{equation}\label{eq:monsterB3}
n(\lambda_1(\mu)+\lambda_2(\mu)-\varepsilon)\leq \log |\det A(x,n)|\leq
n(\lambda_1(\mu)+\lambda_2(\mu)+\varepsilon)
\end{equation}
\begin{equation}\label{eq:monsterB4}
\min_{|\mathtt{k}|\leq n_0}\sum_{i=0}^{n-1}\mathbf{1}_{[\mathtt{k}]}(\sigma^ix) >n_0
\end{equation}
\begin{equation}\label{eq:monsterB5}
\log \mu\left(\left[x_1\cdots x_n\right]\right)<-n\left(h(\mu)-\frac{\varepsilon}{2}\right)
\end{equation}
and
\begin{equation}\label{eq:monsterB6}
\log\|A(x,n)\| < n\left(\lambda_1(\mu)+\varepsilon\right).
\end{equation}
Since $\mu$ is ergodic we have
\[\lim_{n \to \infty} \frac{1}{n}\sum_{i=0}^{n-1}\mu(\sigma^{-i}Z \cap Z)=\mu(Z)^2,\]
so in particular we have for infinitely many $n$
\[\mu(\sigma^{-n}Z\cap Z)>\frac{2}{3}\mu(Z)^2=2\kappa.\]
In particular, for infinitely many $n$,  $\mu(\sigma^{-n}Z \cap Z \cap Y_n)>\kappa$.  For the remainder of the proof we fix an integer $n$ such that $\mu(\sigma^{-n}Z \cap Z \cap Y_n)>\kappa$ and such that additionally
\begin{equation}\label{eq:monsterB7}
e^{-\frac{n\varepsilon}{2}}<\kappa,
\end{equation}
\begin{equation}\label{eq:monsterB8}
e^{-n\varepsilon}<\delta,
\end{equation}
and
\begin{equation}\label{eq:monsterB9}
1-e^{n(\lambda_2(\mu)-\lambda_1(\mu)+\varepsilon)}>e^{-\frac{n\varepsilon}{2}}.
\end{equation}
Let $X:=\sigma^{-n}Z \cap Z \cap Y_n$ and define
\[\Gamma:=\left\{\mathtt{i} \in \{1,\ldots,M\}^n \colon \mu([\mathtt{i}]\cap X)>0\right\}.\]
Clearly $\Gamma$ is nonempty.

We now demonstrate that $\Gamma$ has the properties required in the statement of the proposition, beginning with those which are most easily established. We first estimate the cardinality of $\Gamma$. Clearly
\[\mu\left(\bigcup_{\mathtt{i}\in \Gamma} [\mathtt{i}]\right) \geq \mu(X)>\kappa>e^{-\frac{n\varepsilon}{2}},\]
using \eqref{eq:monsterB7}, and it follows from \eqref{eq:monsterB5} that
\[\mu([\mathtt{i}])<e^{-n\left(h(\mu)-\frac{\varepsilon}{2}\right)}\]
for all $\mathtt{i}\in \Gamma$. Combining these observations yields
\[e^{-\frac{n\varepsilon}{2}}< \sum_{\mathtt{i}\in \Gamma}\mu([\mathtt{i}]) \leq e^{-n\left(h(\mu)-\frac{\varepsilon}{2}\right)}\#\Gamma\]
which is to say $\#\Gamma> e^{n(h(\mu)-\varepsilon)}$, and we have established \eqref{it:pm1}.

It follows from \eqref{eq:monsterB6} that for all $\mathtt{i}\in\Gamma$ we have $\|A_{\mathtt{i}}\|\leq e^{n(\lambda_1(\mu)+\varepsilon)}$, and by the definition of joint spectral radius this implies that the  joint spectral radius of $\{A_{\mathtt{i}}\colon \mathtt{i}\in\Gamma\}$ is at most $e^{n(\lambda_1(\mu)+\varepsilon)}$, which is \eqref{it:pm3}. In view of \eqref{eq:monsterB3} we have $e^{n(\lambda_1(\mu)+\lambda_2(\mu)-\varepsilon)}\leq |\det A_{\mathtt{i}}|\leq e^{n(\lambda_1(\mu)+\lambda_2(\mu)+\varepsilon)}$ which is \eqref{it:pm5}. We may also easily establish \eqref{it:pm6}: given a word $\mathtt{k}$ of length $n_0$ and a word $\mathtt{i}=i_1\cdots i_n \in \Gamma$, there exists $x \in [\mathtt{i}]\cap X$. Using \eqref{eq:monsterB4} there exists an integer $i$ such that $0 \leq i <n-n_0$ and $\mathbf{1}_{[\mathtt{k}]}(\sigma^ix)=1$, and this shows that $\mathtt{k}$ is a subword of $\mathtt{i}$ as claimed.

It remains to bound from below the growth of vectors in $\mathcal{K}$ and to show that the matrices $\{A_{\mathtt{i}}\colon \mathtt{i}\in\Gamma\}$ strictly preserve a cone, establishing points \eqref{it:pm4} and \eqref{it:pm2} respectively. We claim that
\begin{equation}\label{eq:monsterC1}\|A(x,n)u\| \geq e^{n\left(\lambda_1(\mu)-\varepsilon\right)}\|u\|\end{equation}
when $\overline{u} \in \mathcal{K}$ and $x \in X$. To see this we note that $\mathfrak{u}(x)\neq \mathfrak{s}(x)$, and therefore we may find unit vectors $v_u \in \mathfrak{u}(x)$, $v_s\in\mathfrak{s}(x)$ and real numbers $\beta,\gamma$ such that $u=\beta v_u + \gamma v_s$. We observe that
\[|\beta| = \frac{\|u \wedge v_s\|}{\|v_u \wedge v_s\|}\geq \|u\wedge v_s\|\geq \tau\|u\|\]
and
\[|\gamma| = \frac{\|u \wedge v_u\|}{\|v_u\wedge v_s\|} \leq \frac{\|u\|}{\|v_u\wedge v_s\|}\leq \frac{\|u\|}{\tau}\]
using the defining property of $\tau$ together with the fact that $\mathfrak{s}(x) \in \mathcal{J}$ and $\mathfrak{u}(x), \overline{u}\in \mathcal{K}$. We deduce that
\begin{align*}\|A(x,n)u\| &\geq \tau \|A(x,n)v_u\| - \frac{1}{\tau} \|A(x,n)v_s\|\\
&\geq \left(e^{n\left(\lambda_1(\mu)-\frac{\varepsilon}{2}\right)} -  e^{n\left(\lambda_2(\mu)+\frac{\varepsilon}{2}\right)}\right)\|u\|\\
&=e^{n\left(\lambda_1(\mu)-\frac{\varepsilon}{2}\right)} \left(1-e^{n(\lambda_2(\mu)-\lambda_1(\mu)+\varepsilon)}\right)\|u\|\\
&\geq e^{n\left(\lambda_1(\mu)-\varepsilon\right)} \|u\|\end{align*}
using \eqref{eq:monsterB1}, \eqref{eq:monsterB2} and \eqref{eq:monsterB9}, which proves the claim. Given $\mathtt{i} \in \Gamma$, applying the claim to any $x \in [\mathtt{i}] \cap X$ establishes \eqref{it:pm4}, since in this case $A(x,n)=A_{\mathtt{i}}$.

Now let  $\overline{u} \in \mathcal{K}$ and $x \in X \subset Z$. We may estimate
\begin{align*}
d\left(\overline{A(x,n)u},\mathfrak{u}(T^nx)\right)&=d\left(\overline{A(x,n)u},\overline{A(x,n)v_u}\right)&\\
&=\frac{\|A(x,n)u\wedge A(x,n)v_u\|}{\|A(x,n)u\|\cdot \|A(x,n)v_u\|} &\\
&\leq \frac{|\det A(x,n)| \cdot \|u\wedge v_u\|}{e^{n(2\lambda_1(\mu)-2\varepsilon)}\|u\|\cdot\|v_u\|} & (\text{by } \eqref{eq:monsterB1} )\\
&\leq e^{n(\lambda_2(\mu)-\lambda_1(\mu)+3\varepsilon)} d(\overline{u},\mathfrak{u}(x)) & (\text{by } \eqref{eq:monsterB3})\\
&\leq e^{-n\varepsilon} d(\overline{u},\mathfrak{u}(x))<\delta & (\text{by } \eqref{eq:monster0}, \eqref{eq:monsterB8}).
\end{align*}
Since $x \in X$ we have  $T^nx \in Z$ so that $d(\mathfrak{u}(T^nx),\overline{w_u})\leq \delta$, and therefore
\[d(\overline{A(x,n)u},\overline{w_u})\leq d(\overline{A(x,n)u},A(x,n)\mathfrak{u}(x)) +d(\mathfrak{u}(T^nx),\overline{w_u})<2\delta.\]
We have shown in particular that if $x \in X$ and $u\in\overline{u}\in \mathcal{K}$ then $A(x,n)\overline{u}\in \mathrm{Int} \mathcal{K}$. It follows that for any given $\mathtt{i}\in \Gamma$, if $x \in X \cap [\mathtt{i}]$, then the matrix $A(x,n)=A_{\mathtt{i}}$ maps $\mathcal{K}$ into the interior of $\mathcal{K}$, and we have proved \eqref{it:pm2}. The proof of the proposition is complete.\end{proof}
To obtain the full strength of Theorem \ref{th:monster} from the above proposition we require several further lemmas:
\begin{lemma}
Let $A_1,\ldots,A_M$, $\mu$, $n_0$ and $\varepsilon>0$ be as in the statement of Proposition \ref{pr:submonster}. Then the set $\Gamma$ in the conclusion of Proposition \ref{pr:submonster} may be chosen such that for every $A_{\mathtt{i}}$ we have $\det A_{\mathtt{i}}>0$, and $A_{\mathtt{i}}$ strictly preserves a cone $\mathcal{C}$ not depending on $\mathtt{i}\in\Gamma$.
\end{lemma}
\begin{proof}
Let $\Gamma$ be the set constructed by Proposition \ref{pr:submonster} with $\varepsilon/2$ in place of $\varepsilon$, and with $n$ chosen large enough that $e^{-n\varepsilon/2}<\frac{1}{16}$. We will find an integer $n'>n$ and a set $\Gamma'\subseteq \{1,\ldots,M\}^{n'}$ such that all of the conclusions of Proposition \ref{pr:submonster} hold, and such that $\det A_{\mathtt{i}}>0$ for all $\mathtt{i}\in \Gamma'$. Let $\mathcal{K}$ be the projective interval in Proposition \ref{pr:submonster}, and let $\mathcal{C}_1$, $\mathcal{C}_2$ denote the two connected components of the set $\{u \in \mathbb{R}^2 \setminus \{0\} \colon \overline{u}\in\mathcal{K}\}$.  For each $\mathtt{i} \in \Gamma$, by linearity we either have $A_{\mathtt{i}}\mathcal{C}_i\subset \mathcal{C}_i$ for $i=1,2$, or $A_{\mathtt{i}}\mathcal{C}_{3-i}\subset \mathcal{C}_i$ for $i=1,2$.

Choose a subset $\Gamma_0$ of $\Gamma$ such that $\#\Gamma_0 \geq \frac{1}{4}\#\Gamma$, such that $\det A_{\mathtt{i}}$ has the same sign for every $\mathtt{i}\in\Gamma_0$, and either such that every matrix $A_{\mathtt{i}}$ preserves both of the two cones $\mathcal{C}_i$, or such that every matrix $A_{\mathtt{i}}$ interchanges the two cones $\mathcal{C}_i$. Define $\Gamma':=\{\mathtt{i}\mathtt{j}\colon \mathtt{i},\mathtt{j}\in \Gamma_0\}$ and $n':=2n$. Clearly for every $\mathtt{i} \in \Gamma'$ we have $\det A_{\mathtt{i}}>0$ and $A_{\mathtt{i}}$ maps $\mathcal{C}_1$ into its own interior. Clearly
\[\#\Gamma' \geq \frac{1}{16} \#\Gamma  \geq \frac{1}{16}e^{-2n(h(\mu)-\frac{\varepsilon}{2})} >e^{-2n(h(\mu)-\varepsilon)}=e^{-n'(h(\mu)-\varepsilon)}\]
using Proposition \ref{pr:submonster}\eqref{it:pm1}, and $\Gamma'$ may be easily seen to inherit all of the other properties listed in Proposition \ref{pr:submonster} as required.
\end{proof}
The above lemma completes the proof of the Theorem in the case where $(A_i)_{i=1}^M$ are not assumed to be strongly irreducible. Before treating the strongly irreducible case, we require an additional lemma:
\begin{lemma}\label{le:pablo}
Let $A_1,\ldots,A_M\in GL_2(\mathbb{R})$ be strongly irreducible, and let $\overline{u},\overline{s} \in \mathbb{RP}^1$ be the unstable and stable directions of a hyperbolic matrix $A_{\mathtt{i}}$. Then there exist $m \geq 1$ and $\mathtt{k}\in\{1,\ldots,M\}^m$ such that $A_{\mathtt{k}}\overline{u},A_{\mathtt{k}}\overline{s} \notin \{\overline{u},\overline{s}\}$.\end{lemma}
\begin{proof}

Firstly, we claim that, as a consequence of strong irreducibility, there exists $\mathtt{j}$ such that $\overline{s}\notin \{ A_{\mathtt{j}}\overline{u}, A_{\mathtt{j}}\overline{s}\}$. Indeed, suppose this is not the case. By strong irreducibility, there exist words $\mathtt{j}_1,\mathtt{j}_2$ such that $A_{\mathtt{j}_i}\overline{s}\neq \overline{s}$ (so that $A_{\mathtt{j}_i}\overline{u}=\overline{s}$), and $A_{\mathtt{j}_1}\overline{s}\neq A_{\mathtt{j}_2}\overline{s}$. Let $A_p$ be any matrix which does not fix $\overline{s}$. Then $A_p A_{\mathtt{j}_i}\overline{u}=A_p \overline{s}\neq\overline{s}$, so we must have $A_p A_{\mathtt{j}_i}\overline{s}=\overline{s}$ for $i=1,2$. This contradicts the injectivity of the action of $A_p$ on $\mathbb{RP}^1$.

On the other hand, by strong irreducibility there exists $\mathtt{j}'$ such that $A_{\mathtt{j}'}\overline{u}\notin\{\overline{u},\overline{s}\}$. Since $A_{\mathtt{i}}$ is hyperbolic we have $\lim_{\ell \to \infty} A_{\mathtt{i}}^\ell A_{\mathtt{j}}\overline{u}=\lim_{\ell \to \infty} A_{\mathtt{i}}^\ell A_{\mathtt{j}}\overline{s}=\overline{u}$ and therefore $\lim_{\ell \to \infty} A_{\mathtt{j}'}A_{\mathtt{i}}^\ell A_{\mathtt{j}}\overline{u}=A_{\mathtt{j}'}A_{\mathtt{i}}^\ell A_{\mathtt{j}}\overline{s}=A_{\mathtt{j}'}\overline{u}\notin \{\overline{u},\overline{s}\}$. It follows that if $\ell$ is sufficiently large then $\mathtt{k} := \jj' \ii^\ell \jj$ satisfies $A_{\mathtt{k}}\overline{u},A_{\mathtt{k}}\overline{s} \notin \{\overline{u},\overline{s}\}$ as desired.
\end{proof}
The remaining case is dealt with by the following lemma:
\begin{lemma}
Let $A_1,\ldots,A_M$, $\mu$, $n_0$ and $\varepsilon>0$ be as in the statement of Theorem \ref{th:monster}, and suppose that $\Gamma \subset \{1,\ldots,M\}^n$ satisfies all of the conclusions of Theorem \ref{th:monster} except possibly \eqref{it:tm5}, and with $\varepsilon/2$ in place of $\varepsilon$. If $(A_1,\ldots,A_M)$ is strongly irreducible, then there exist $n'>n$ and $\Gamma'\subset \{1,\ldots,M\}^{n'}$ for which all of the conclusions of Theorem \ref{th:monster} are satisfied.
\end{lemma}
\begin{proof}Let $\mathcal{C} \subset \mathbb{R}^2$ be a cone which is strictly preserved  by every element of $\{A_{\mathtt{i}}\colon \mathtt{i}\in \Gamma\}$, and let $\mathcal{K}$ denote the projective image of $\mathcal{C}$. We note that each $A_{\mathtt{i}}$ contracts $\mathcal{K}$ with respect to the angle-sine metric $d$, and therefore the projective transformation $A_{\mathtt{i}}$ has a unique fixed point in  $\mathcal{K}$ which is an attractor for the projective transformation. In particular, the unstable eigenspace of every $A_{\mathtt{i}}$ lies in $\mathcal{K}$, and the stable eigenspace of every $A_{\mathtt{i}}$ does not lie in $\mathcal{K}$.

Choose $\mathtt{i} \in \Gamma$ arbitrarily, and note that since $A_{\mathtt{i}}$ strictly preserves the cone $\mathcal{C}$, it is hyperbolic by virtue of the Perron-Frobenius Theorem. Let $\overline{u}, \overline{s}\in \mathbb{RP}^1$ be respectively the unstable and stable eigenspaces of $A_{\mathtt{i}}$. By Lemma \ref{le:pablo} there exist an integer $m_0$ and a finite word $\mathtt{k} \in \{1,\ldots,M\}^{m_0}$, which in general will not belong to $\Gamma$, such that $A_{\mathtt{k}}\overline{u} \notin \{\overline{u},\overline{s}\}$ and $A_{\mathtt{k}}\overline{s}\neq\overline{s}$. Since clearly
\[\bigcap_{m=1}^\infty A_{\mathtt{k}}A_{\mathtt{i}}^m\mathcal{K} = \left\{A_{\mathtt{k}}\overline{u}\right\}\]
and this sequence of sets is nested, we may choose an integer $m_1 \geq 1$ such that $A_{\mathtt{k}}A_{\mathtt{i}}^{m_1}\mathcal{K}$ does not intersect $\{\overline{u},\overline{s}\}$. In a similar manner, if $m_2$ is sufficiently large then $A_{\mathtt{i}}^{m_2}A_{\mathtt{k}}A_{\mathtt{i}}^{m_1}\mathcal{K}$ is contained in the interior of $\mathcal{K}$. Choose $m_2$ with this property. Now let $m_3$ be an integer which is large enough that additionally
\begin{align*}e^{(m_0+nm_1+nm_2+nm_3)(\lambda_1(\mu)-\varepsilon)} &\leq \det A_{\mathtt{i}}^{m_1+m_2+m_3}|\det A_{\mathtt{k}}|\\
& \leq e^{(m_0+nm_1+nm_2+nm_3)(\lambda_1(\mu)+\varepsilon)},\end{align*}
where we have used Theorem \ref{th:monster}\eqref{it:tm4}. If $m_3$ is sufficiently large then it is also clear that
\begin{align*}\|A_{\mathtt{i}}^{m_2+m_3}A_{\mathtt{k}}A_{\mathtt{i}}^{m_1}\| &\leq e^{n(m_1+m_2+m_3)(\lambda_1(\mu)+\frac{\varepsilon}{2})} \left(\max_{1 \leq i \leq M}\|A_i\|\right)^{m_0}\\
 &\leq e^{(m_0+nm_1+nm_2+nm_3)(\lambda_1(\mu)+\varepsilon)}\end{align*}
using \eqref{it:tm4}, and using \eqref{it:tm2b}, if $m_3$ is sufficiently large then for all $\overline{u}\in\mathcal{K}$ we have
\begin{align*}\|A_{\mathtt{i}}^{m_2+m_3}A_{\mathtt{k}}A_{\mathtt{i}}^{m_1}u\| &\geq \|A_{\mathtt{i}}^{m_3}\|\alpha_2(A_{\mathtt{i}}^{m_2}A_{\mathtt{k}}A_{\mathtt{i}}^{m_1})\|u\|\\
&\geq e^{nm_3(\lambda_1(\mu)-\frac{\varepsilon}{2})}\left(\min_{1\leq i \leq M}\alpha_2(A_i)\right)^{m_0+nm_1+nm_2} \|u\|\\&\geq e^{(n(m_1+m_2+m_3)+m_0)(\lambda_1(\mu)-\varepsilon)}\|u\|,\end{align*}
where we have used the fact that $A_{\mathtt{i}}^{m_2}A_{\mathtt{k}}A_{\mathtt{i}}^{m_1}\overline{u}\in \mathcal{K}$ and the fact that $A_{\mathtt{i}}$ preserves  $\mathcal{K}$. The matrix $A_{\mathtt{i}}^{m_2+m_3}A_{\mathtt{k}}A_{\mathtt{i}}^{m_1}$ either has positive determinant, or negative determinant. Clearly it maps $\mathcal{K}$ into the interior of $\mathcal{K}$, and consequently it either maps $\mathcal{C}$ to the interior of $\mathcal{C}$, or to the interior of $-\mathcal{C}$. In any event, $(A_{\mathtt{i}}^{m_2+m_3}A_{\mathtt{k}}A_{\mathtt{i}}^{m_1})^2$ has positive determinant and maps the cone $\mathcal{C}$ to its own interior.

Define now $k:=2(m_0+nm_1+nm_2+nm_3)$, $n':=nk$ and $\mathtt{j}=(\mathtt{i}^{m_2+m_3}\mathtt{k}\mathtt{i}^{m_1})^{2n} \in \{1,\ldots,M\}^{n'}$. We claim that $A_{\mathtt{j}}$ does not have any eigenspaces in common with $A_{\mathtt{i}}^k$. Indeed, if $A_{\mathtt{i}}^k\overline{v}=A_{\mathtt{j}}\overline{v}=\overline{v}$ then $\overline{v}$ must equal either $\overline{u}$ or $\overline{s}$. In the former case we have $(A_{\mathtt{i}}^{m_2+m_3}A_{\mathtt{k}}A_{\mathtt{i}}^{m_1})^{2n}\overline{u}=\overline{u}$.  This matrix strictly preserves a cone and hence is hyperbolic by the Perron-Frobenius theorem; we deduce $A_{\mathtt{i}}^{m_2+m_3}A_{\mathtt{k}}A_{\mathtt{i}}^{m_1}\overline{u}=\overline{u}$. Since $\overline{u}$ is invariant for $A_{\mathtt{i}}$ we obtain $A_{\mathtt{i}}^{m_2+m_3}A_{\mathtt{k}}\overline{u}=\overline{u}$, and since $\overline{u}$ is invariant for $A_{\mathtt{i}}^{-1}$ we obtain $A_{\mathtt{k}}\overline{u}=\overline{u}$, contradicting the definition of $\mathtt{k}$.  The equation $A_{\mathtt{j}}\overline{s}=\overline{s}$ leads to the contradiction $A_{\mathtt{k}}\overline{s}=\overline{s}$ in an identical manner. Let us now define \[\Gamma':=\left\{\mathtt{i}_1\cdots \mathtt{i}_k \colon \mathtt{i}_j\in \Gamma\right\} \cup \{\mathtt{j}\}\subset \{1,\ldots,M\}^{n'}.\]

We have seen that for every  $\mathtt{l}\in \Gamma'$ the cone $\mathcal{C}$ is mapped to its own interior by $A_{\mathtt{l}}$, and using the estimates proved above it is easy to check that $\Gamma'$ satisfies the properties stipulated in Theorem \ref{th:monster} with $n'$ in place of $n$. To see that $\{A_{\mathtt{l}}\colon \mathtt{l}\in \Gamma'\}$ is strongly irreducible, we note that this set contains the two matrices $A_{\mathtt{i}}^k,A_{\mathtt{j}}$ which are hyperbolic and do not have a common invariant subspace. In particular, if any $\overline{u} \in \mathbb{RP}^1$ is given then $\overline{u}$ either is not fixed by $A_{\mathtt{i}}^k$ or is not fixed by $A_{\mathtt{j}}$, and so its orbit $(A_{\mathtt{i}}^{k})^n\overline{u}$ (resp. $A_{\mathtt{j}}^n\overline{u}$) is infinite and cannot be contained in a finite union of subspaces.
\end{proof}

\section{From strong open set condition to strong separation}
\label{se:SOSCtoSSC}

Let $T=(T_1,\ldots,T_M)$ be strict contractions on $\R^d$, and let $E$ be associated invariant set, i.e. $E$ is compact, nonempty and $E=\cup_i T_i(E)$. We recall some standard notions of separation:
\begin{itemize}
 \item $T$ is said to satisfy the \emph{strong separation condition} (SSC) if $T_i E\cap T_j E=\emptyset$ whenever $i\neq j$.
 \item $T$ satisfies the \emph{strong open set condition} (SOSC) if there exists a nonempty bounded open set $U$ with $U\cap E\neq\emptyset$, such that $T_i U\subset U$ and $T_i(U) \cap T_j(U)=\emptyset$ for all $i\neq j$.
 \item $T$ satisfies the \emph{open set condition} (OSC) if there exists a nonempty bounded open set $U$, such that $T_i U\subset U$ and $T_i(U) \cap T_j(U)=\emptyset$ for all $i\neq j$.
\end{itemize}
It is easy to see that SSC$\Rightarrow$SOSC$\Rightarrow$OSC. The OSC and SOSC are known to be equivalent when $T_i$ are similarities, but this equivalence breaks down in the self-affine case: Edgar \cite[Example 1]{Ed92} constructed a non-trivial affine IFS satisfying the OSC, for which all of the maps have the same fixed point, so that the attractor degenerates to this fixed point. The SOSC clearly cannot hold, since any open set containing the common fixed point cannot be mapped into disjoint sets by the IFS. Although this argument, and Edgar's construction, are fairly simple, we have not been able to find this observation in the literature.

The following result will allow us to deduce results for self-affine sets satisfying the SOSC from results which are known to hold only under the SSC.
\begin{theorem} \label{th:SOSCtoSSC}
Let $T=(T_1,\ldots,T_M)$ be a finite set of invertible affine contractions on $\R^2$ which satisfies the strong open set condition, with $T_i(x)=A_i(x)+v_i$. Let $\mu$ be a $\sigma$-invariant measure on $\Sigma_M$. Suppose $(A_1,\ldots,A_M)$ and $\mu$ satisfy the assumptions of Theorem \ref{th:monster}.

Then for any $\varepsilon>0$, there exist $n$ and a subset $\Gamma\subset\{1,\ldots, M\}^n$ satisfying all the conditions of Theorem \ref{th:monster} (except (vii)) and, in addition, $T_\Gamma=(T_\jj:\jj\in\Gamma)$ satisfies the strong separation condition.
\end{theorem}

Note that although the OSC is trivially preserved when passing to subsystems of iterates (the same open set works), things are less clear for the SOSC, as the open set may stop intersecting the new, smaller attractor. The following simple lemma will allow us to overcome this issue.
\begin{lemma} \label{le:SOSC-subsystem}
If $T=(T_1,\ldots,T_M)$ satisfy the SOSC, then there exist $n_0$ and a word $\ii_0\in \{1,\ldots,M\}^{n_0}$ with the following property: if $\Gamma$ is a subset of $\{1,\ldots,M\}^{n_1}$ where $n_1\ge n_0$, such that $\ii_0$ appears as a subword of some word of $\Gamma$, then $\{ T_\jj:\jj\in\Gamma\}$ also satisfies the SOSC (with the same open set).
\end{lemma}
\begin{proof}
Let $U$ be the open set for $T$. It follows from the definition of SOSC that there exist $n_0$ and $\ii_0\in \{1,\ldots,M\}^{n_0}$ such that $T_{\ii_0}(\overline{U})\subset U$. Let $\Gamma$ be as in the statement of the lemma and suppose $\jj=(\kk \ii_0 \kk')\in\Gamma$ (where $\kk$ or $\kk'$ might be the empty word). Clearly $\{ T_\jj:\jj\in\Gamma\}$ satisfies the OSC with the same open set $U$. Moreover, since
\[
T_\jj(\overline{U})=T_{\kk \ii_0 \kk'} (\overline{U}) \subset T_\kk T_{\ii_0}\overline{U}\subset T_{\kk}{U}\subset U \subset \overline{U},
\]
it follows that the fixed point of the contraction $T_\jj$ belongs to $U$. Since this point belongs to the attractor the SOSC is satisfied.
\end{proof}

The following lemma will help us achieve strong irreducibility of the new subsystem.
\begin{lemma} \label{le:irreducibility-with-matrix-at-front}
Let $A_1,\ldots,A_M \in GL_2(\mathbb{R})$ be hyperbolic matrices which do not have a common invariant subspace. Let $B \in GL_2(\mathbb{R})$. Then for infinitely many $n\geq 1$, the set
\[\left\{A_{i_1}\cdots A_{i_n}B \colon 1\leq i_1,\ldots,i_n \leq M\right\}\]
is irreducible.
\end{lemma}
\begin{proof}
We prove the lemma by contradiction. If the conclusion is false then there exists a sequence $(\overline{v}_n)$ of elements of $\mathbb{RP}^1$ such that for all large enough $n$ we have $A_{i_1}\cdots A_{i_n}B\overline{v}_n=\overline{v}_n$ for all $i_1,\ldots,i_n\in\{1,\ldots,M\}$.
Since the matrices $A_1,\ldots,A_M$ are irreducible, at least two of them are not scalar multiples of one another. Without loss of generality, we assume that $A_1$ is not a scalar multiple of $A_2$. We have $A_1^{n-1}A_2B\overline{v}_n=\overline{v}_n=A_1^nB\overline{v}_n$ for all large enough $n$, so in particular $A_1^{-1}A_2B \overline{v}_n=B\overline{v}_n$ for all large enough $n$. Since $A_1^{-1}A_2$ is not a scalar multiple of the identity it fixes at most two elements of $\mathbb{RP}^1$, and this implies that the sequence $(B \overline{v}_n)$ can take at most two distinct values when $n$ is sufficiently large.  We may therefore choose $\overline{u}\in\mathbb{RP}^1$ and a strictly increasing sequence of natural numbers $(n_r)_{r=1}^\infty$ such that $B\overline{v}_{n_r}=\overline{u}$ for all $r\geq 1$. In particular
\[A_1^{n_r}\overline{u}=A_2^{n_r}\overline{u}=\cdots =A_M^{n_r}\overline{u}\]
for all $r\geq 1$. Since the matrices $A_i$ are hyperbolic, for each $i$ the sequence $A_i^{n_r}\overline{u}$ converges projectively as $r\to\infty$ to an invariant subspace of $A_i$. Taking the limit $r\to\infty$ in the above equation we conclude that there exists a common invariant subspace of $A_1,\ldots,A_M$, which is a contradiction.
\end{proof}

\begin{proof}[Proof of Theorem \ref{th:SOSCtoSSC}]
We first choose $n_0$ and $\ii_0$ as in Lemma \ref{le:SOSC-subsystem}. Next, we choose $n_1$ and a subset $\Gamma'\subset \{1,\ldots,M\}^{n_1}$ satisfying the conditions of Theorem \ref{th:monster} for the given value of $\e$. Hence, we know from Lemma \ref{le:SOSC-subsystem} that $T_{\Gamma'}=( T_\jj:\jj\in\Gamma')$ satisfies the SOSC, so we can pick $m$ and $\ii_1\in\{1,\ldots,M\}^{m n_1}$ such that $T_{\ii_1}(\overline{U})\subset U$, where $U$ is the corresponding open set.

Let $m'$ be a sufficiently large integer to be determined later. Write $n=m' n_1+m n_1$ and
\[
\Gamma=\left\{ \kk\ii_1:\kk\in(\Gamma')^{m'} \right\} \subset \{1,\ldots,M\}^n.
\]
The IFS $T_\Gamma:=(T_\jj:\jj\in\Gamma)$ satisfies the SSC. Indeed, pick $\jj_1\neq\jj_2\in \Gamma$. We can write $\jj_i=\kk \jj'_i\ii_1$ for some words $\jj'_i$ starting with different symbols $a_i$. Hence
\[
T_{\jj_1}(\overline{U})\cap T_{\jj_2}(\overline{U}) = T_{\kk}\left( T_{\jj'_1} T_{\ii_1} \overline{U}\cap T_{\jj'_2} T_{\ii_1} \overline{U} \right) \subset T_\kk(T_{a_1}U\cap T_{a_2}U) =\emptyset.
\]

We claim that $m'$ can be taken so that $\Gamma$ satisfies all the conditions of Theorem \ref{th:monster}, with $O(\e)$ in place of $\e$ (which is obviously enough to establish the claim). Note that the topological entropy of the subsystem is
\[
m' \log(\Gamma') > m'n_1(h(\mu)-\e) > n(h(\mu)-2\e),
\]
provided $m'$ is taken large enough. A similar calculation shows that parts (iii) and (iv) hold with $O(\e)$ in place of $\e$ if $m'$ is sufficiently large. Note that the implicit constant depends on $\mu$, but this does not matter as $\e$ is arbitrary.

Part (ii) is obvious, and if the original matrices $A_i$ were not strongly irreducible then this completes the proof. Otherwise it remains to establish strong irreducibility of $\{A_\ii \colon \ii \in \Gamma\}$. As all matrices $A_\ii$ are hyperbolic, we only need to show irreducibility. However, this follows from Lemma \ref{le:irreducibility-with-matrix-at-front}, provided $m'$ was taken from the infinite set provided by that lemma.

\end{proof}

\subsection{The case $s\ge 2$ under the OSC}

The next lemma is standard but we include the proof for completeness. It shows that in Theorems \ref{th:Lyap32} and \ref{th:HueterLalley}, the only non-trivial  case is that in which the affinity dimension is strictly less than $2$.

\begin{lemma} \label{le:affin-ge-2-OSC}
 Let $E$ be the invariant set under the affinities $(T_1,\ldots,T_M)$.
 \begin{enumerate}
  \item If $\adim(T_1,\ldots,T_M)=2$ and the OSC holds, then $E$ has non-empty interior (in particular, Hausdorff dimension $2$).
  \item If $\adim(T_1,\ldots,T_M)>2$, then the OSC cannot hold.
 \end{enumerate}
\end{lemma}
\begin{proof}
Suppose $\adim(T_1,\ldots,T_M)=2$ and the OSC holds with open set $U$. Since $\adim(T_1,\ldots,T_M)=2$, we have $\sum_{i=1}^M \det(T_i)=1$, so $(T_i(U))_{i=1}^M$ is a partition of $U$ in measure. By iterating, so is $\{ T_\ii(U): \ii\in \{1,\ldots,M\}^n\}$ for any $n$. This implies that
\[
 U \subset \bigcap_{n=1}^\infty \bigcup_{\ii\in \{1,\ldots,M\}^n} \overline {T_\ii(U)} = \bigcap_{n=1}^\infty \bigcup_{\ii\in \{1,\ldots,M\}^n} T_\ii(\overline {U}) = E,
\]
giving the first claim.

Next, observe that $\adim(T_1,\ldots,T_M)>2$ if and only if $\sum_{i=1}^M \det(T_i) > 1$. If the OSC holds with bounded open set condition $U$, then $T_i(U)$ are pairwise disjoint subsets of $U$ whose area adds up to $(\sum_{i=1}^M \det(T_i))$ times the area of $U$, which cannot happen if $\adim(T_1,\ldots,T_M)>2$.
\end{proof}

\subsection{Proof of Theorem \ref{th:approx-by-Bernoulli}}

We can now easily conclude the proof of Theorem \ref{th:approx-by-Bernoulli}. By Theorem \ref{th:properties-of-mu}, the K\"{a}enm\"{a}ki measure has different Lyapunov exponents, so Theorem \ref{th:monster} (and, for the last claim, Theorem \ref{th:SOSCtoSSC}) is applicable. Hence, fix $\e>0$, and let $n, \Gamma$ be as given by Theorems \ref{th:monster} or \ref{th:SOSCtoSSC}.

The only claim in Theorem \ref{th:approx-by-Bernoulli} which is not immediate is the first one. Let $\nu$ be the uniform Bernoulli measure on $\Gamma^\mathbb{N}$. It follows from Theorem \ref{th:monster}(i) that
\[
h(\nu) \ge n(h(\mu)-\e),
\]
from Theorem \ref{th:monster}(iii) that
\[
\lambda_1(\nu) \in (n(\lambda_1(\mu)-\e),n(\lambda_1(\mu)+\e)),
\]
and from Theorem \ref{th:monster}(iv) that
\[
\lambda_1(\nu)+\lambda_2(\nu) \in (n(\lambda_1(\mu)+\lambda_2(\mu)-\e),n(\lambda_1(\mu)+\lambda_2(\mu)+\e)),
\]
which, combined with the previous observation, yields
\[
\lambda_2(\nu) \in (n(\lambda_2(\mu)-2\e),n(\lambda_2(\mu)+2\e)).
\]
The definition of Lyapunov dimension then implies that there exists a constant $C=C(\lambda_1(\mu),\lambda_2(\mu))>0$ such that
\[
\ldim(\nu,(A_\ii)_{\ii\in\Gamma}) \ge \ldim(\mu,A)-C\e = \adim(A)-C\e.
\]
Since $\e$ is arbitrary, this concludes the proof of Theorem \ref{th:approx-by-Bernoulli}.

\subsection{A counterexample to Theorems \ref{th:Lyap32} and \ref{th:HueterLalley} under the OSC}

\label{subsec:counterexample}

\begin{example} \label{ex:counterex-32}

Define eight matrices $A_1,\ldots,A_8 \in GL_2(\mathbb{R})$ by

\begin{align*}
A_1&:=\left(\begin{array}{cc}\frac{1}{8}&0\\\frac{1}{2}&\frac{1}{2} \end{array}\right),\qquad
A_2:=\left(\begin{array}{cc}\frac{1}{4}&\frac{1}{8}\\\frac{1}{2}&\frac{1}{2} \end{array}\right),\\
A_3&:=\left(\begin{array}{cc}\frac{3}{8}&\frac{1}{4}\\\frac{1}{2}&\frac{1}{2} \end{array}\right),\qquad
A_4:=\left(\begin{array}{cc}\frac{1}{2}&\frac{3}{8}\\\frac{1}{2}&\frac{1}{2} \end{array}\right),\\
A_5&:=\left(\begin{array}{cc}\frac{1}{2}&\frac{1}{2}\\ \frac{3}{8}&\frac{1}{2} \end{array}\right),\qquad
A_6:=\left(\begin{array}{cc}\frac{1}{2}&\frac{1}{2}\\\frac{1}{4}&\frac{3}{8} \end{array}\right),\\
A_7&:=\left(\begin{array}{cc}\frac{1}{2}&\frac{1}{2}\\\frac{1}{8}&\frac{1}{4} \end{array}\right),\qquad
A_8:=\left(\begin{array}{cc}\frac{1}{2}&\frac{1}{2}\\0&\frac{1}{8} \end{array}\right).
\end{align*}
and define $T_1,\ldots,T_8 \colon \mathbb{R}^2 \to \mathbb{R}^2$ by $T_ix:=A_ix$ for each $i=1,\ldots,8$. Then $(T_1,\ldots,T_8)$ satisfies all the hypotheses of Theorem \ref{th:Lyap32} except that the OSC holds instead of the SOSC.
\end{example}
\begin{proof}
It is straightforward to check that each $T_i$ is a contraction which fixes $0$, and it follows that the attractor of $(T_1,\ldots,T_8)$ is simply $\{0\}$. Clearly $A_1$ and $A_8$ are hyperbolic and it is easily checked that they do not share an eigenspace, so the matrices $A_1,\ldots,A_8$ are irreducible. One may also verify that the transformations $T_i$ satisfy the OSC with open set $U:=(0,1)^2$ (see Figure \ref{fi:bosc-not-sosc}). Finally, since
\begin{align*}
\lim_{n\to \infty} \frac{1}{n}\log \left(\sum_{|\mathbf{i}|=n}\varphi^{\frac{3}{2}}(A_{\mathbf{i}})\right) &\geq \lim_{n\to \infty} \frac{1}{n}\log \left(\sum_{|\mathbf{i}|=n}\left|\det A_{\mathbf{i}}\right|^{\frac{3}{4}}\right)\\
&=\log \sum_{i=1}^8 |\det A_i|^{\frac{3}{4}}=\log \sum_{i=1}^8 16^{-\frac{3}{4}}=0,
\end{align*}
the affinity dimension of $(T_1,\ldots,T_N)$ is at least $\frac{3}{2}$.

The SOSC cannot hold since any open set intersecting the attractor contains the fixed point of all maps (of course, failure of the SOSC also follows from Theorem \ref{th:Lyap32}).
\end{proof}

\setlength{\unitlength}{4cm}

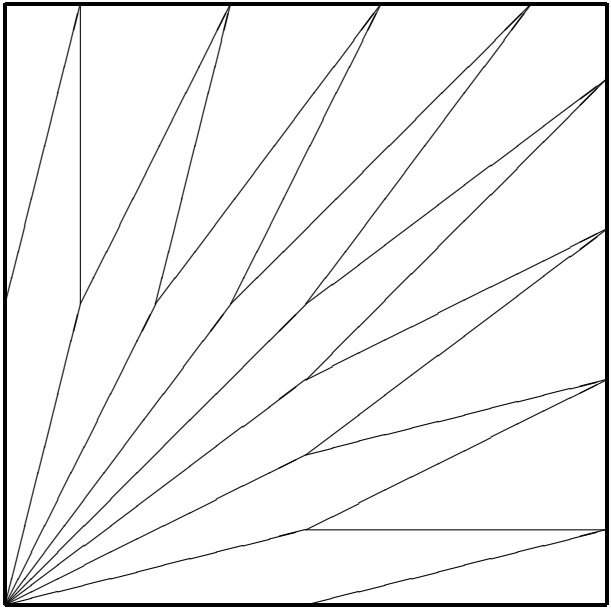
\begin{figure}\label{fi:bosc-not-sosc}

\begin{picture}(2,2)




\put(0,0){\line(1,4){0.25}}

\put(0,0){\line(1,2){0.5}}

\put(0,0){\line(3,4){0.75}}

\put(0,0){\line(1,1){1}}

\put(0,0){\line(4,3){1}}

\put(0,0){\line(2,1){1}}

\put(0,0){\line(4,1){1}}


\put(0,1){\line(1,4){0.25}}

\put(0.25,1){\line(1,2){0.5}}

\put(0.5,1){\line(3,4){0.75}}

\put(0.75,1){\line(1,1){1}}


\put(0.25,1){\line(0,1){1}}

\put(0.5,1){\line(1,4){0.25}}

\put(0.75,1){\line(1,2){0.5}}

\put(1,1){\line(3,4){0.75}}


\put(1,0){\line(4,1){1}}

\put(1,0.25){\line(2,1){1}}

\put(1,0.5){\line(4,3){1}}

\put(1,0.75){\line(1,1){1}}


\put(1,0.25){\line(1,0){1}}

\put(1,0.5){\line(4,1){1}}

\put(1,0.75){\line(2,1){1}}

\put(1,1){\line(4,3){1}}

\linethickness{1.2pt}


\put(0,0){\line(0,1){2}}

\put(0,0){\line(1,0){2}}

\put(2,0){\line(0,1){2}}

\put(0,2){\line(1,0){2}}

\end{picture}

\caption{Each of the eight rhombuses in the diagram is an image of the open unit square $(0,1)^2$ under one of the eight maps $T_i$ in Example \ref{ex:counterex-32}. The boundary of the diagram is the boundary of the unit square $[0,1]^2$.}

\end{figure}

The counterexample to Theorem \ref{th:HueterLalley} is similar but easier; it is enough to modify Edgar's example, \cite[Example 1]{Ed92}, so that the bunching assumption is met. For example, one can take $T_i(x)=A_i x$, where
\[
A_1:=\left(\begin{array}{cc}\frac{1}{5}&\frac{1}{5}\\0&\frac{1}{5} \end{array}\right),\qquad
A_2:=\left(\begin{array}{cc}\frac{1}{5}&0\\\frac{1}{5}&\frac{1}{5} \end{array}\right).
\]

\section{Applications}
\label{se:applications}

\subsection{Review of relevant results}

Here we present some recent advances in the dimension theory of self-affine systems, which we shall need in the proofs of our main applications. All of these results involve the Furstenberg measure associated to a Bernoulli measure on $\Sigma_M$ and a tuple $A=(A_1,\ldots,A_M)\in GL_2(\R)^M$. This is the push-down of the natural extension of $\mu$ under the unstable direction map $\mathfrak{u}(x)$ given by Theorem \ref{th:Oseledets}. Concretely, given an ergodic invariant measure $\mu$ on $\Sigma_M$ with invertible natural extension $\hat\mu$, for our purposes the Furstenberg measure $\eta=\eta_\mu$ is the Borel probability measure on $\RP^1$ defined by
\[
 \eta(B) = \hat\mu\left(\{x:\mathfrak{u}(x)\in B\}\right)
\]
where $\mathfrak{u}$ is given by the application of Theorem \ref{th:Oseledets} to $\hat\mu$.
This is well-defined whenever $\mu$ has different Lyapunov exponents, which will always be the case whenever we speak of a Furstenberg measure, even if $\mu$ is not a Bernoulli measure. We underline that in the Bernoulli case other definitions exist, but they are equivalent to the above one when the $A_i$ are strongly irreducible and the generated subgroup contains a hyperbolic matrix.

Recall that the \emph{(lower) Hausdorff dimension of a measure} $\mu$ is defined as
\[
\dim_H\mu = \inf \{ \dim_H(A): \mu(A)>0\}.
\]
A measure $\mu$ on $\R^2$ (or more generally any metric space) is said to be \emph{exact dimensional} if there exists $s$ (called the exact dimension of $\mu$) such that
\[
\lim_{r\searrow 0} \frac{\log \mu(B(x,r))}{\log r} = s
\]
for $\mu$-almost all $x$. Many measures of dynamical origin are exact dimensional, although this is often a highly nontrivial fact. By $\dim\mu=s$ we will mean that $\mu$ has exact dimension $s$. In this case, the Hausdorff dimension of the measure agrees with $s$. In particular, if $\dim\mu=s$ and $\mu(A)>0$, then $\dim_H(A)\ge s$.

Very recently, B\'{a}r\'{a}ny and K\"{a}enm\"{a}ki \cite{BaKa15} proved that every self-affine measure on the plane is exact dimensional, and its exact dimension can be expressed in terms of the so-called Ledrappier-Young formula. Previously, B\'{a}r\'{a}ny \cite{Ba15} and Falconer and Kempton \cite{FaKe15a} had established some special cases. We quote a result from \cite{Ba15}; although it is less general than the results from \cite{BaKa15}, its proof is simpler and it is enough for our purposes.

\begin{theorem}[{\cite[Theorem 2.8]{Ba15}}] \label{th:Barany}
Let $A=(A_1,\ldots,A_M) \in GL_2(\mathbb{R})$  be a set of contracting matrices strictly preserving a cone (or more generally satisfying dominated splitting), and let $\mu$ be a Bernoulli measure on $\Sigma_M$. Suppose that
\[
\dim_H(\eta_\mu)\ge \min(1,\ldim(\mu)).
\]
Then for every set of translations $v=(v_1,\ldots,v_M)$ such that $T_v=(A_i x+v_i)_{i=1}^m$ satisfies the SSC, the corresponding self-affine measure $\nu_v$ is exact-dimensional, and
\[
\dim(\nu_v) = \ldim(\mu,A).
\]
\end{theorem}

B\'{a}r\'{a}ny \cite[Theorem 2.9]{Ba15} also proved equality of the dimension of self-affine measures $\nu$ and Lyapunov dimension when $\dim(\nu)+\dim(\eta_\mu)\ge 2$. A drawback of this result is that it requires a-priori lower estimates for $\dim\nu$. A. Rapaport \cite{Ra15} was able to replace $\dim(\nu)$ by $\ldim(\mu,A)$, under some very mild condition on the matrices:

\begin{theorem}[{\cite[Main theorem]{Ra15}}] \label{th:Rapaport}
Let $A=(A_1,\ldots,A_M) \in GL_2(\mathbb{R})^M$ be an irreducible set of matrices, and suppose that $\mu$ is a Bernoulli measure on $\Sigma_M$ with different Lyapunov exponents, and such that
\[
\ldim(\mu) + \dim_H(\eta_\mu) > 2,
\]
where $\eta_\mu$ is the Furstenberg measure induced by $\mu$. Then for every set of translations $v=(v_1,\ldots,v_M)$ such that $T_v=(A_i x+v_i)_{i=1}^M$ satisfies the SSC, the self-affine measure $\nu_v$ induced by $\mu$ and $T_v$ satisfies
\[
\dim(\nu_v) = \ldim(\mu, A).
\]
\end{theorem}
We recall that Furstenberg's Theorem \cite{Fu63} guarantees that if the $A_i$ are strongly irreducible and the generated semigroup contains a hyperbolic matrix, then any Bernoulli measure has different Lyapunov exponents, so the above theorem is applicable.

Note that the dimension of the Furstenberg measure plays a key r\^{o}le in both of the last theorems. We conclude this review with a result of Hochman and Solomyak which provides a new condition under which the dimension of the Furstenberg measure is the ``expected'' one.
\begin{theorem}[{\cite[Theorem 1.1]{HoSo15}}] \label{th:Hochman-Solomyak}
Let $A_1,\ldots, A_M\in SL_2(\R)$ be a strongly irreducible set of matrices with exponential separation, whose generated semigroup contains a hyperbolic matrix.

Then for any Bernoulli measure $\mu$ on $\Sigma_M$, if we denote by $\eta=\eta_\mu$ the corresponding Furstenberg measure on $\RP^1$, then
\[
\dim\eta = \min\left(1,\frac{h(\mu)}{2\lambda_1(\mu)}\right).
\]
\end{theorem}
We note that the setting of \cite{HoSo15} allows for non-freely generated groups if one can estimate the random walk entropy, but this does not seem to be helpful for our applications, so we stick to the simpler situation above.

\subsection{A consequence of Theorem \ref{th:Hochman-Solomyak}}

We will need to apply Theorem \ref{th:Hochman-Solomyak} in the form given by the following corollary.
\begin{corollary} \label{co:Hochman-Solomyak}
 Let $A_1,\ldots, A_M\in GL_2^+(\R)$ be strongly irreducible matrices with exponential separation whose generated semigroup is not compact.

Then for any Bernoulli measure $\mu$ on $\Sigma_M$, if we denote by $\eta=\eta_\mu$ the corresponding Furstenberg measure on $\RP^1$, then
\[
\dim\eta = \min\left(1,\frac{h(\mu)}{\lambda_1(\mu)-\lambda_2(\mu)}\right).
\]
\end{corollary}
\begin{proof}
For $A\in GL_2^+(\R)$, let $\overline{A}=\det{A}^{-1/2}A$. Then $\overline{A}_i\in SL_2(\R)$. We claim that
\[
 \|\overline{A}_\ii-\overline{A}_\jj\| \ge \tfrac12 \delta^n
\]
if $\ii\neq\jj\in \{1,\ldots,M\}^n$, where
\[
 \delta = \frac{c^2}{\max_{i} \det(A_i)\max_i \|\overline{A}_i\| }.
\]
In other words, $\overline{A}_i$ also has exponential separation. Indeed, if  $|\overline{A}_\ii-\overline{A}_\jj|<\tfrac12 \delta^n$ for some $\ii,\jj\in\{1,\ldots,M\}^n$, then
\begin{align*}
 c^{2n} &\le \|A_\ii A_\jj- A_\jj A_\ii \|\\
 &= \det(A_\ii)^{1/2}\det(A_\jj)^{1/2}   \|\overline{A}_\ii \overline{A}_\jj- \overline{A}_\jj \overline{A}_\ii \| \\
 &\le \det(A_\ii)^{1/2}\det(A_\jj)^{1/2} \left( \|\overline{A}_\ii \overline{A}_\jj - \overline{A}_\ii^2\|  +\| \overline{A}_\ii^2 - \overline{A}_\jj \overline{A}_\ii\| \right) \\
 &\le 2  \det(A_\ii)^{1/2} \det(A_\jj)^{1/2}  \| \overline{A}_\ii \| \|\overline{A}_\ii-\overline{A}_\jj\| \\
 &< (\max_{i} \det(A_i))^n (\max_i \|\overline{A}_i\|)^n \delta^n,
\end{align*}
contradicting the choice of $\delta$.

On the other hand, if $\lambda_1(\mu)>\lambda_2(\mu)$ are the Lyapunov exponents for the cocycle generated by the $A_i$ then, since
\[
\log \|\overline{A}_{x_1}\cdots \overline{A}_{x_n}\| = \log \det(A_{x_1}\cdots A_{x_n})^{-1/2} +\log\|A_{x_1}\cdots A_{x_n}\|,
\]
the top Lyapunov exponent for the cocycle $\overline{A}_i$ is
\[
-\frac{\lambda_1(\mu)+\lambda_2(\mu)}{2} + \lambda_1(\mu) = \frac{\lambda_1(\mu)-\lambda_2(\mu)}{2}.
\]
The conclusion now follows from Theorem \ref{th:Hochman-Solomyak}.
\end{proof}

\subsection{Proof of Theorem \ref{th:Lyap32} and generalizations}

Theorem \ref{th:Lyap32} will follow as a corollary of the following more general result.
\begin{theorem} \label{th:application-Rapaport}
Let $(T_1,\ldots, T_M)$ be strictly contractive, invertible affine maps, with $T_i(x)=A_i x+v_i$, and let $E$ be the associated self-affine set. Suppose that the following conditions hold:
\begin{enumerate}
\item
The transformations $(A_1,\ldots,A_M)$ are strongly irreducible and generate a semigroup which contains a hyperbolic matrix.
\item
The affinities $(T_1,\ldots,T_M)$ satisfy the strong open set condition.
\item
The maps $(A_1,\ldots,A_M)$ have exponential separation.
\item
\[
\adim(T)+\dim_S(\eta)>2,
\]
where $\eta$ is the Furstenberg measure induced by the K\"aenm\"aki measure $\mu$ for $A$, and
\[
\dim_S\eta = \min\left(1,\frac{h(\mu)}{\lambda_1(\mu)-\lambda_2(\mu)}\right)
\]
is the similarity dimension of $\eta$.
\end{enumerate}
Then $\dim_H(E)=\adim(T_1,\ldots,T_M)$.
\end{theorem}
\begin{proof}
Let $A_{\ii_0}$ be hyperbolic with $|\ii_0|=m$. By replacing $T_1,\ldots,T_M$ with the $M^m$ transformations $T_{i_1}\cdots T_{i_m}$ if necessary, we may assume without loss of generality that $|\ii_0|=1$. One can readily check that this iteration does not affect any of the hypothesis of the theorem; in particular, exponential separation is preserved (with a different constant $c$). Now apply Theorem \ref{th:approx-by-Bernoulli} with a sufficiently small $\e>0$ to obtain $n,\Gamma,\nu$ as in that theorem.  Since exponential separation is also preserved when passing to subsystems, it holds in particular for $( A_\ii: \ii\in\Gamma)$. Hence we can apply Corollary \ref{co:Hochman-Solomyak} to the Furstenberg measure $\eta_\nu$ associated to $\nu$ and  $( A_\ii: \ii\in\Gamma)$ to obtain
\[
\dim\eta_\nu = \min\left(1,\frac{h(\nu)}{\lambda_1(\nu)-\lambda_2(\nu)}\right) \ge \min\left(1,\frac{h(\mu)-\e}{\lambda_1(\mu)-\lambda_2(\mu)+2\e}\right).
\]
Hence, provided $\e$ was chosen sufficiently small,
\[
\ldim(\nu, (A_\ii:\ii\in\Gamma)) + \dim\eta_\nu > 2.
\]
Since $(T_\ii:\ii\in\Gamma)$ satisfies the SSC by Theorem \ref{th:approx-by-Bernoulli}, we conclude from Theorem \ref{th:Rapaport} that
\[
 \dim_H(E) \ge \dim\nu = \ldim(\nu, A),
\]
which can be taken arbitrarily close to $\adim(A_1,\ldots,A_M)$. Since the opposite inequality always holds, this completes the proof.
\end{proof}

We can now deduce Theorem \ref{th:Lyap32} as a corollary. In fact, we will weaken the required bound on the affinity dimension in terms of the bunching behavior of the maps $A_i$.
\begin{theorem} \label{th:Lyap32-generalized}
Suppose $(A_1,\ldots,A_M)$ satisfy assumptions (1)-(3) of Theorem \ref{th:Lyap32} and, furthermore, $\alpha_1(A_i) \le \alpha_2(A_i)^t$ for some $t\in [0,1/2)$ and all $i=1,\ldots,M$, and
\[
\adim(A_1,\ldots,A_M) \ge \frac{3(1-t)}{2-t}.
\]
Then, for any $v=(v_1,\ldots,v_M)$ such that $(A_i x+v_i)$ satisfies the strong open set condition, the associated self-affine set $E_v$ satisfies $\dim_H(E_v)=\adim(A_1,\ldots,A_M)$.
\end{theorem}
Note that Theorem \ref{th:Lyap32} follows immediately by taking $t=0$. Also, if $t\ge 1/2$, then Theorem \ref{th:HueterLalley}, which has no a priori assumption on the affinity dimension, becomes applicable. We also point out that for $t<1/2$, the lower bound on the affinity dimension is always larger than $1$.

In order to deal with the endpoint in the proof of Theorem \ref{th:Lyap32-generalized}, we will require the following lemma. It will allow us to show that a non-strict bunching condition is enough to guarantee a strict inequality between the Lyapunov exponents.
\begin{lemma}
Let $\mathcal{S}\subset GL_2(\mathbb{R})$ be a semigroup of contractions such that $\alpha_1(A)=\alpha_2(A)^t$ for every $A \in\mathcal{S}$ and some $t\in (0,1)$. Then the elements of $\mathcal{S}$ are simultaneously diagonalisable.
\end{lemma}
\begin{proof}
We observe that $\alpha_1(AB)=\alpha_1(A)\alpha_1(B)$ for all $A,B\in \mathcal{S}$, since
\begin{align*}
\alpha_1(AB)^{1+1/t}&=\alpha_1(AB)\alpha_2(AB)=|\det AB|=|\det A|\cdot|\det B|\\
&=\alpha_1(A)^{1+1/t}\alpha_1(B)^{1+1/t}.
\end{align*}
  Since $\alpha_1(A)^{1+1/t}=\alpha_2(A) \in (0,1)$, the singular values of every $A \in \mathcal{S}$ must be distinct.
Let us suppose firstly that $\mathcal{S}\subset GL_2^+(\mathbb{R})$. Fix $A_1 \in \mathcal{S}$ and let $A_2 \in \mathcal{S}$ be arbitrary; we will find a basis depending only on $A_1$ in which both matrices are diagonal.  Let $R_\theta$ denote the matrix of rotation through angle $\theta$. Taking singular value decompositions we may write $A_i=R_{\psi_i}D_iR_{\phi_i}$ for $i=1,2$, where $D_i$ is a positive diagonal matrix with entries equal to the singular values of $A_i$, listed in decreasing order down the diagonal. We have
\[\|D_i\|^2=\alpha_1(A_i)^2=\alpha_1(A_i^2)=\|A_i^2\|=\|R_{\psi_i}D_iR_{\phi_i}R_{\psi_i}D_iR_{\psi_i}\|=\|D_iR_{\phi_i}R_{\psi_i}D_i\|\]
which, since the entries of $D_i$ are distinct, is only possible if $R_{\phi_i}R_{\psi_i}$ is plus or minus the identity. Since similarly
\[\|D_1\|\cdot\|D_2\|=\alpha_1(A_1)\alpha_1(A_2)=\alpha_1(A_1A_2)=\|A_1A_2\|=\|D_1R_{\phi_1}R_{\psi_2}D_2\|\]
we must have $R_{\phi_1}R_{\psi_2}=\pm\mathrm{Id}=\pm R_{\phi_1}R_{\psi_1}=\pm R_{\phi_2}R_{\psi_2}$. We note the particular consequence $R_{\phi_2}=\pm R_{\phi_1}$. We deduce from these identities that
\[R_{\phi_1}A_1R_{\phi_1}^{-1} = R_{\phi_1}R_{\psi_1}D_1=\pm D_1,\]
\[R_{\phi_1}A_2R_{\phi_1}^{-1} = R_{\phi_1}R_{\psi_2}D_2R_{\phi_2}R_{\phi_1}^{-1}=\pm D_2\]
so that $A_1$ and $A_2$ are simultaneously diagonal, and moreover are hyperbolic. Since $R_{\phi_1}$ and $R_{\psi_1}$ depend only on $A_1$ it follows that $\mathcal{S}$ is simultaneously diagonalisable as claimed, and furthermore all of its elements are hyperbolic.

Now suppose that $\mathcal{S}\setminus GL_2^+(\mathbb{R})$ is nonempty. Applying the above argument we may find a basis in which every element of the semigroup $\mathcal{S}\cap GL_2^+(\mathbb{R})$ is diagonal and hyperbolic. In particular if $A \in \mathcal{S}\setminus GL_2^+(\mathbb{R})$ then $A^2 \in \mathcal{S}\cap GL_2^+(\mathbb{R})$. If the square of a $2\times 2$ matrix is diagonal and hyperbolic then so must be the original matrix, and it follows that in this basis every $A \in \mathcal{S}$ is diagonal as claimed.
\end{proof}

\begin{corollary} \label{co:bunching-to-Lyap-exponents}
Suppose $(A_1,\ldots,A_M)\in GL_2(\mathbb{R})^M$ is irreducible. If $\mu$ is an ergodic, fully supported measure on $\Sigma_M$, and if $\alpha_1(A_i)^\tau\leq \alpha_2(A_i)$ for every $i=1,\ldots,M$ and some $\tau>1$, then $\tau\lambda_1(\mu)<\lambda_2(\mu)$.
\end{corollary}
\begin{proof}
The sequence
\[
\int_{\Sigma_M} \tau\log\alpha_1(A_{x_n}\cdots A_{x_1}) -\log \alpha_2(A_{x_n}\cdots A_{x_1}) \,d\mu(x)
\]
is subadditive and bounded above by $0$, so its limit $\tau\lambda_1(\mu)-\lambda_2(\mu)$ is negative if and only if there exists an integer $n$ such that the above integral is negative. Since $\mu$ is fully supported, this occurs if and only if there exists an element $A$ of the semigroup such that $\alpha_1(A)^\tau<\alpha_2(A)$. Since the semigroup is irreducible the existence of such an element follows from the previous lemma.
\end{proof}

\begin{proof}[Proof of Theorem \ref{th:Lyap32-generalized}]
In light of Theorem \ref{th:application-Rapaport}, it is enough to show that, under the assumptions of the theorem,
\[
\adim(A_1,\ldots,A_M) + \mathrm{dim}_{S}\eta_\mu>2,
\]
where $\mu$ is the K\"aenm\"aki  measure, $\eta_\mu$ the corresponding Furstenberg measure. The claim is trivial if $\dim_S\eta_\mu=1$, so in the following we will assume $\dim_S\eta_\mu<1$.

We will suppose the conclusion to be false and deduce a strict upper bound of  $\tfrac{3(1-t)}{2-t}$ for the affinity dimension, which is a contradiction. Let $s\geq 1$ denote the affinity dimension. In light of Lemma \ref{le:affin-ge-2-OSC} we may assume that $s\in [1,2)$. The K\"aenm\"aki measure $\mu$ is an equilibrium state for $\varphi^s$ and therefore satisfies
\[h(\mu)+\lim_{n\to\infty} \frac{1}{n}\int \log \varphi^s(A(x,n))d\mu(x)=P(\varphi^s,A)\]
which is to say
\[h(\mu)+\lambda_1(\mu)+(s-1)\lambda_2(\mu)=0,\]
using the definition of the affinity dimension and the fact that $s\in [1,2)$. Hence,
\begin{equation}\label{eq:entropy}
h(\mu)=-\lambda_1(\mu)+(1-s)\lambda_2(\mu).
\end{equation}
By hypothesis we have
\[
s+\frac{h(\mu)}{\lambda_1(\mu)-\lambda_2(\mu)} \leq 2
\]
which is to say
\[
s\lambda_1(\mu)-s\lambda_2(\mu)+h(\mu) \leq 2\lambda_1(\mu)-2\lambda_2(\mu).
\]
Substituting in the value for the entropy given in \eqref{eq:entropy} yields
\[
(s-1)\lambda_1(\mu)+(1-2s)\lambda_2(\mu) \leq 2\lambda_1(\mu)-2\lambda_2(\mu)
\]
or equivalently
\[
(s-3)\lambda_1(\mu) \leq (2s-3)\lambda_2(\mu),
\]
from which we obtain
\[
\frac{\lambda_1(\mu)}{\lambda_2(\mu)} \leq \frac{3-2s}{3-s},
\]
where we note that $\lambda_2(\mu)$ and $s-3$ are both negative.
On the other hand, the assumption $\alpha_1(A_i) \le \alpha_2(A_i)^t$ together with Corollary \ref{co:bunching-to-Lyap-exponents} imply that $\lambda_1(\mu) < t \lambda_2(\mu)$. Recalling that $\lambda_2(\mu)<0$, we deduce that
\[
t < \frac{\lambda_1(\mu)}{\lambda_2(\mu)} \leq \frac{3-2s}{3-s},
\]
from which, solving for $s$, we get
\[
s < \frac{3(1-t)}{2-t},
\]
as desired.
\end{proof}

\subsection{Proof of Theorem \ref{th:HueterLalley}}
The proof of Theorem \ref{th:HueterLalley} is similar, except that we rely on Theorem \ref{th:Barany} instead.
\begin{proof}[Proof of Theorem \ref{th:HueterLalley}]
Let $\mu$ be the K\"{a}enm\"{a}ki measure. We know from Theorem \ref{th:properties-of-mu} that $\mu$ is globally supported, so $2\lambda_1(\mu)<\lambda_2(\mu)$ by Corollary \ref{co:bunching-to-Lyap-exponents}.

On the other hand, it is easy to check that $\ldim(\mu)\le h(\mu)/(-\log\lambda_1(\mu))$ by considering the cases $\ldim(\mu)\in (0,1]$ (in which there is equality), and $\ldim(\mu)>1$ separately. It follows that
\[
\ldim(\mu) < \frac{h(\mu)}{\lambda_1(\mu)-\lambda_2(\mu)} =: \tau.
\]

Now given a sufficiently small $\e>0$,  let $n,\Gamma,\nu$ be as provided by Theorem \ref{th:approx-by-Bernoulli}. Let $\eta_\nu$ the Furstenberg measure corresponding to $\nu$ and $(A_\ii:\ii\in\Gamma)$. If $\tau>1$, then arguing as in the proof of Theorem \ref{th:application-Rapaport}, by choosing $\e$ small enough we can ensure that $\dim\eta_\nu=1$. The claim then follows from Theorems \ref{th:approx-by-Bernoulli} and \ref{th:Barany} by letting $\e\to 0$. Otherwise, by picking $\e$ small enough we can ensure that
\[
\ldim(\nu) < \frac{h(\nu)}{\lambda_1(\nu)-\lambda_2(\nu)} = \dim(\eta_\nu),
\]
where the last equality follows from Corollary \ref{co:Hochman-Solomyak}. Hence Theorem \ref{th:Barany} is still applicable and the claim follows from Theorem \ref{th:approx-by-Bernoulli} by letting $\e\to 0$.
\end{proof}

\subsection{Projections of self-affine sets}

The problem of computing the dimension of projections of dynamically defined sets and measures has received a great deal of attention in the last decade, and the situation is fairly well understood in the self-similar setting, see e.g. \cite{Sh15, ShSo15} and references there. In the self-affine case, some results were obtained in the carpet case \cite{FeJoSh10, FeFrSa15}, but it was only very recently that Falconer and Kempton proved a result for projections of self-affine measures in a more general situation \cite{FaKe15b}. Their main results \cite[Theorem 3.1 and Corollary 3.2]{FaKe15b} hold for self-affine measures under the assumption that all matrices $A_i$ are strictly positive. In combination with the results in this article, we obtain:

\begin{theorem} \label{th:projections}
Suppose $T_i(x)=A_i x+v_i$, $i=1,\ldots,M$, satisfy the assumptions of either Theorem \ref{th:Lyap32}, Theorem \ref{th:Lyap32-generalized}, or Theorem \ref{th:HueterLalley}, and let $E$ be the associated self-affine set. Then for any linear map $P:\R^2\to \R$,
\[
\dim_H(PE) =\min(\dim_H(E),1).
\]
\end{theorem}
\begin{proof}
In the course of the proof of Theorems \ref{th:Lyap32-generalized} and \ref{th:HueterLalley}, it is shown that $\dim_H(E)$ (which equals the affinity dimension of $A$) can be approximated by $\ldim(\mu,(A_\ii)_{\ii\in\Gamma})$, where $\mu$ and $\Gamma\subset\{1,\ldots,M\}^n$ are given by Theorem \ref{th:SOSCtoSSC}. In particular, $(A_\ii)_{\ii\in\Gamma}$ is irreducible and, after a change of coordinates, all the $A_\ii,\ii\in\Gamma$ are strictly positive. Moreover, applying either Theorem \ref{th:Barany} or Theorem \ref{th:Rapaport}, we know that the self-affine measure $\nu$ corresponding to the system $(A_\ii)_{\ii\in\Gamma}$ and the Bernoulli measure $\mu$ satisfies
\[
\dim\nu = \ldim(\mu,(A_\ii)_{\ii\in\Gamma}).
\]
Since $(A_\ii)_{\ii\in\Gamma}$ is irreducible, the set $B$ appearing in \cite[Corollary 3.2]{FaKe15b} equals all of $\RP^1$. Therefore, we conclude from \cite[Corollary 3.2]{FaKe15b} that for any linear map $P:\R^2\to\R$,
\[
\dim_H(PE) \ge \dim(P\nu) = \min(\dim\nu,1) = \min(\ldim(\mu,(A_\ii)_{\ii\in\Gamma}),1).
\]
Since $\ldim(\mu,(A_\ii)_{\ii\in\Gamma})$ can be made arbitrarily close to $\dim_H(E)$, and the inequality $\dim_H(PE)\le \min(\dim_H E,1)$ holds for any set $E\subset\R^2$ and Lipschitz map $P$, the claim follows.
\end{proof}

\subsection{Concrete examples}\label{subs:concrete}
\label{ss:examples}

We now show how to apply our previous results to obtain many new explicit classes of self-affine sets for which Hausdorff and affinity dimensions coincide.  To the best of our knowledge, all previously known such examples fall into at least one of the following categories:
\begin{itemize}
  \item the maps are similarities,
  \item the maps are simultaneously diagonalizable,
  \item the maps strictly preserve a cone.
\end{itemize}
In many of our examples, the generated semigroup include both an elliptic and a hyperbolic matrix, so they provide genuinely new examples of equality of Hausdorff and affinity dimension.

We begin with two easy well-known lemmas which will help us verify the exponential separation condition in Theorems \ref{th:Lyap32} and \ref{th:HueterLalley}. We write $\text{sg}(A)$ for the semigroup generated by $A\subset GL_2(\R)$. By abuse of notation, given $A \in GL_2(\mathbb{R})$ we shall also write $A$ for the transformation of $\mathbb{RP}^1$ induced by $A$.
\begin{lemma} \label{le:free-sg-cones}
Let $A=(A_1,\ldots,A_M)\in GL_2(\mathbb{R})^M$. If there exists a nonempty set $K\subset \mathbb{RP}^1$ such that $A_i(K)\subset K$ and such that $A_i(K)\cap A_j(K)=\emptyset$ for all $i\neq j$, then the $A_i$ freely generate $\text{sg}(A)$.
\end{lemma}
\begin{proof}
Suppose $A_\ii=A_\jj$ with $\ii\neq \jj$. By replacing $\ii $ with $\ii\jj$ and $\jj$ with $\jj \ii$ if necessary we may assume the words $\ii$ and $\jj$ to have equal length. Since $\ii \neq \jj$ we may write $\ii= \kk \lll$, $\jj = \kk \lll'$ where $\lll_1 \neq \lll_1'$ (with the length of $\kk$ possibly being zero). We then have $A_\lll = A_{\lll'}$ and  $\lll_1 \neq \lll_1'$. For $x\in K$ we have $A_\lll(x)\in A_{\lll_1}(K)$ and $A_{\lll'}(x)\in A_{\lll_1'}(K)$, contradicting $A_\lll(x)=A_{\lll'}(x)$.
\end{proof}

\begin{lemma} \label{le:free-sg-conjugates}
Let $P_{ij}^{(k)}\in \mathbb{Q}[x]$ for $1\le i,j\le 2$, $1\le k\le M$. For $t\in\R$, let
\[
A_k(t) = \left(P_{ij}^{(k)}(t)\right)_{ij},
\]
\[
A(t) = \{ A_k(t) \}_{k=1}^M.
\]
If $t$ is an algebraic number such that $\text{sg}(A(t))$ is freely generated by $A_1(t),\ldots,A_m(t)$, then $\text{sg}(A(u))$ is freely generated by $A_1(u),\ldots,A_m(u)$, for any Galois conjugate $u$ of $t$.
\end{lemma}
\begin{proof}
Given distinct finite words $\ii,\jj$, the function $x\mapsto A_\ii(x)-A_\jj(x)$ is a polynomial with rational coefficients which does not vanish at $t$, hence it does not vanish at $u$ either.
\end{proof}

The above lemmas imply that, for a given $M\ge 2$, the set
\[
\{A=(A_1,\ldots,A_M)\in GL_2(\R)^M: \text{sg}(A) \text{ is free } \}
\]
is dense in $GL_2(\R)^M$. Indeed, we can start with an arbitrary $B=(B_1,\ldots,B_M)$ robustly generating a free semigroup (it is easy to construct examples with the help of Lemma \ref{le:free-sg-cones}), and apply Lemma \ref{le:free-sg-conjugates} with suitable quadratic $t,u$ and linear $P_{ij}^{(k)}$. To see this, note that e.g. $\{ (a+b\sqrt{2},a-b\sqrt{2}):a\in\mathbb{Q},b\in\mathbb{Q}^+\}$ is dense in $\R^2$. In particular, one can easily construct examples satisfying the conditions of the following corollary.
\begin{corollary} \label{co:examples}
  Let $(A_1,\ldots,A_M)\in GL_2(\R)$ have algebraic coefficients, and freely generate a free semigroup which contains both an elliptic and a hyperbolic matrix. Suppose $(r_1,\ldots,r_M)$ are such that either:
  \begin{enumerate}
    \item[(a)] $\adim(r_1 A_1,\ldots, r_M A_M)\in [3/2,2]$, or
    \item[(b)] $|r_i| \alpha_1(A_i)^2\le \alpha_2(A_i)$ for all $i$.
  \end{enumerate}

  Then, for every $v_1,\ldots,v_M$ such that $(T_i(x)=r_i A_i x+v_i)_{i=1}^M$ satisfies the strong open set condition, the invariant set $E$ for $T_i$ has equal Hausdorff and affinity dimensions.

  Moreover, $\dim_H(PE)=\min(\dim_H E,1)$ for all linear maps $P:\R^2\to\R$.
\end{corollary}
\begin{proof}
Since the semigroup acts freely and contains an elliptic element, it is strongly irreducible, and it is non-compact thanks to the hyperbolic matrix. Moreover, it is proved in \cite[Lemma 6.1]{HoSo15} that exponential separation holds when the semigroup acts freely and the coefficients are algebraic. As shown in the proof of Corollary \ref{co:Hochman-Solomyak}, the tuples $(r_i A_i)$ also have exponential separation. The corollary then follows from Theorems \ref{th:Lyap32}, \ref{th:HueterLalley} and \ref{th:projections}.
\end{proof}
\begin{figure}\label{fi:one}
    \centering
    \includegraphics[width=0.8\linewidth]{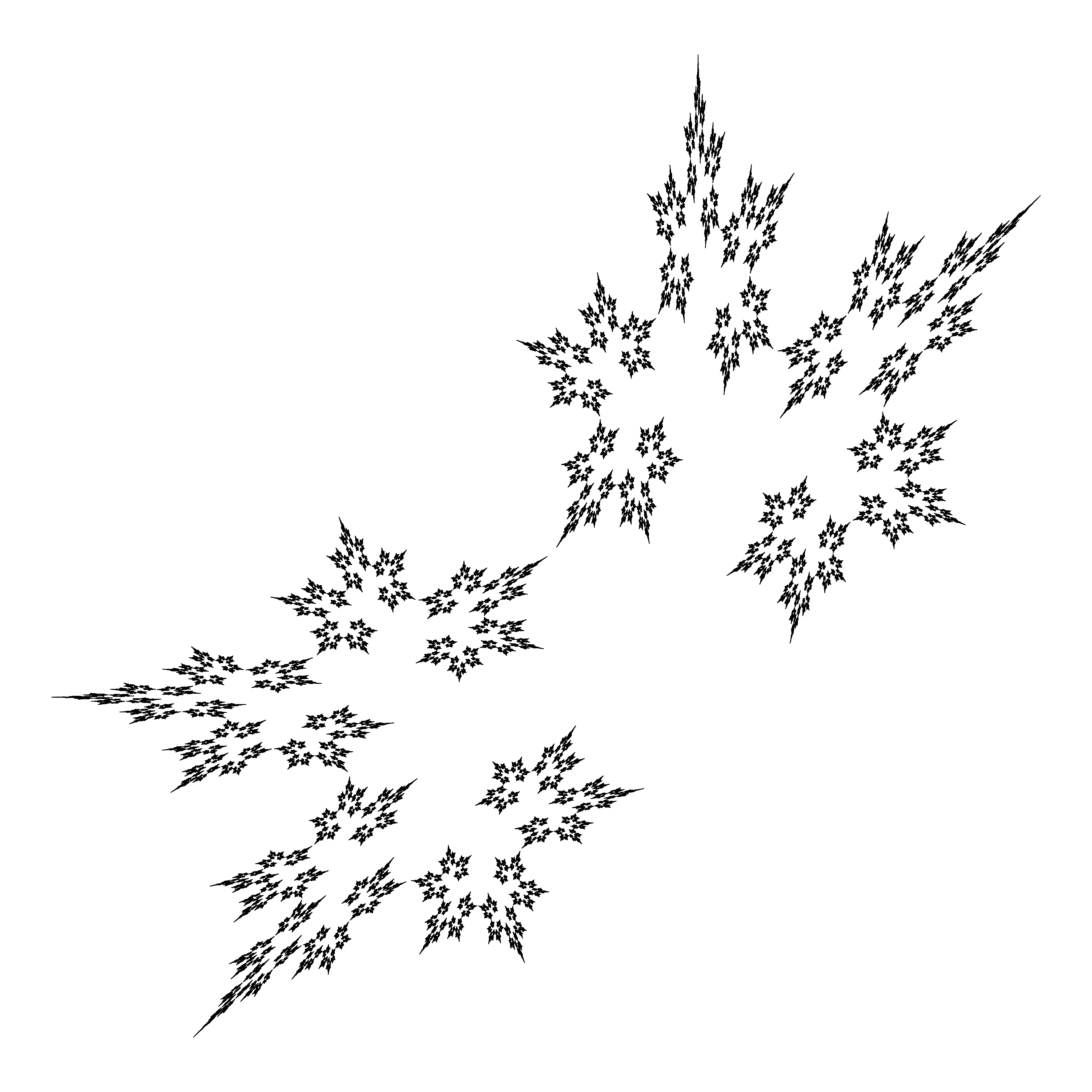}
    \caption{This self-affine set is the attractor of an iterated function system $(T_1,T_2)$ defined at the end of \S\ref{subs:concrete}. It has Hausdorff dimension equal to the affinity dimension of its defining iterated function system. The value of the affinity dimension is unknown, but exceeds $\frac{3}{2}$.}
    \end{figure}
For completeness let us give an explicit example of a pair of elliptic matrices $(A_1,A_2)$ which satisfy the hypotheses of Corollary \ref{co:examples}. Define
\[B_1(t):=\left(\begin{array}{cc}1&t\\1&0\end{array}\right),\qquad B_2(t):=\left(\begin{array}{cc}0&1\\t&1\end{array}\right)\]
for every real number $t$. Obviously we have
\[B_1(\sqrt{2})=\left(\begin{array}{cc}1&\sqrt{2}\\1&0\end{array}\right),\qquad B_2(\sqrt{2})=\left(\begin{array}{cc}0&1\\\sqrt{2}&1\end{array}\right).\]
It is clear that $B_1(\sqrt{2})$ and $B_2(\sqrt{2})$ both preserve the open positive quadrant in $\mathbb{R}^2$ and map that quadrant to two disjoint image cones, one lying above the diagonal in $\mathbb{R}^2$ and the other below it. A simple application of Lemma \ref{le:free-sg-cones} shows that $(B_1(\sqrt{2}),B_2(\sqrt{2}))$ freely generates a free semigroup and hence by Lemma \ref{le:free-sg-conjugates} so does $(B_1(-\sqrt{2}),B_2(-\sqrt{2}))$. Let us therefore define
\[A_1:=B_1(-\sqrt{2})= \begin{pmatrix}1&-\sqrt{2}\\1&0\end{pmatrix}, \qquad A_2:=B_2(-\sqrt{2})=\begin{pmatrix}0&1\\-\sqrt{2}&1\end{pmatrix}\]
and
\[v_1:=\begin{pmatrix}-1\\-1\end{pmatrix},\qquad v_2:=\begin{pmatrix}1\\1\end{pmatrix}.\]
The matrices $A_1$ and $A_2$ each have non-real eigenvalues $\frac{1}{2}\pm \frac{i}{2}\sqrt{4\sqrt{2}-1}$. On the other hand $A_1A_2$ has unequal real eigenvalues $1$ and $2$. The reader may easily check that $A_1$ and $A_2$ both have norm $\sqrt{2+\sqrt{2}}<2^{\frac{11}{12}}$. It follows that if we define two affine transformations of $\mathbb{R}^2$ by $T_ix:=2^{-\frac{11}{12}} A_ix+v_i$ for $i=1,2$ then each $T_i$ is a contraction. The reader may easily verify that $\sum_{i=1}^2 |\det (2^{-\frac{11}{12}}A_i)|^{\frac{3}{4}}=1$ and therefore $\adim (T_1,T_2)\geq \frac{3}{2}$. The pair $(T_1,T_2)$ satisfies the Strong Separation Condition (see Figure \ref{fi:one}) and therefore the hypotheses of Corollary \ref{co:examples} are satisfied.

\section{The irreducible but not strongly irreducible case}
\label{se:irreducible-case}

In this section we work with systems of the form $(A_1,\ldots, A_M)$ such that, for some $\ell\in \{1,\ldots,M-1\}$,
\begin{equation} \label{eq:irr-not-strong-irr}
\begin{aligned}
A_i &= \left(
        \begin{array}{cc}
          a_i & 0 \\
          0 & b_i \\
        \end{array}
      \right)  \quad\text{if } 1\le i \le \ell,\\
A_i &= \left(
        \begin{array}{cc}
          0 & c_i \\
          d_i & 0 \\
        \end{array}
      \right)  \quad\text{if } \ell+1\le i \le M.
\end{aligned}
\end{equation}
Recall from Lemma \ref{le:irr-not-strong-irr} that if $(A_1,\ldots,A_M)$ is irreducible but not strongly irreducible, then after a change of coordinates and re-ordering it does have the above form. It was shown recently in \cite{Mo17} that the affinity dimension of this type of system is remarkably easy to calculate.

Given invertible affine contractions $(T_1,\ldots,T_M)$ with $T_i=A_i x+v_i$, the attractor $E$ is a carpet of the type investigated by J. Fraser in \cite{Fr12}. Note, however, that Fraser only studied the packing and box-counting dimensions of these carpets. Here we investigate their Hausdorff dimension. We underline that, even in the diagonal case, it is well known that the Hausdorff dimension can be strictly smaller than the packing/box-counting and affinity dimensions, even under the strong separation condition. The only known mechanism for this dimension drop is an exact overlap in some coordinate projection.

We start by showing that, as a corollary of our main technical results, one can approximate the affinity dimension by the Lyapunov dimension of Bernoulli measures on a diagonal subsystem.
\begin{proposition} \label{prop:reduction-to-Bernoulli-irr-but-not-strong-irr}
Let $(T_1,\ldots,T_M)$ be invertible affine contractions, where $T_i x=A_i x+v_i$, and the $A_i$ have the form \eqref{eq:irr-not-strong-irr}. Assume also that $\adim(A_1,\ldots,A_M)<2$.

Then for every $\e>0$ there exist $n$, and a set $\Gamma\subset \{1,\ldots,M\}^n$, such that if $\nu$ is the uniform Bernoulli measure on $\Sigma_\Gamma$, then the  following hold:
\begin{enumerate}
\item The matrices $A_\jj, \jj\in\Gamma$, are diagonal, orientation-preserving, and strictly preserve a cone.
\item $\ldim(\nu) \ge \adim(A_1,\ldots,A_M)-\e$.
\item The measure $\nu$ has distinct Lyapunov exponents.
\item If $(T_1,\ldots,T_M)$ satisfies the strong open set condition, then $(T_{\mathtt{j}}:\mathtt{j}\in\Gamma)$ satisfies the strong separation condition.
\end{enumerate}
\end{proposition}
\begin{proof}
Let $\mu$ be the K\"{a}enm\"{a}ki measure for $(A_1,\ldots,A_M)$. We know from Theorem \ref{th:properties-of-mu} that $\mu$ is fully supported and has different Lyapunov exponents. Hence $\mu$  meets the hypothesis of Theorem \ref{th:monster}. Let $\nu$ be the uniform Bernoulli measure on $\Sigma_\Gamma$. It follows from parts (i), (iii) and (iv) of Theorem \ref{th:monster} and a short calculation that
\[
\ldim(\nu) \ge \ldim(\mu) - O(\e) = \ldim(\mu) - O(\e).
\]
If the $T_i$ satisfy the SOSC, we can apply Theorem \ref{th:SOSCtoSSC} to ensure that $(T_{\mathtt{j}}:\mathtt{j}\in \Gamma)$ satisfies the SSC.

To conclude, note that the matrices $A_\jj,\jj\in\Gamma$ must be diagonal since anti-diagonal ones do not preserve a cone, and $\nu$ has different Lyapunov exponents by domination.
\end{proof}

The advantage of the above proposition is that diagonally self-affine sets and measures are much better understood; see \cite{Ba15, BaRaSi15, FrSh15} for some recent advances, most of which rely on Hochman's results \cite{Ho14}. The principal projections play a key r\^{o}le in the diagonal case (since one of them, or both, are atoms for the Furstenberg measure); let $P_x, P_y$ denote projection onto the corresponding coordinate axis. We give two concrete applications of Proposition \ref{prop:reduction-to-Bernoulli-irr-but-not-strong-irr}.

\begin{proposition} \label{pr:irr-not-strong-irr-algebraic}
Let $T_i x=A_i x+v_i$, $i=1,\ldots,M$, be invertible affine contractions, with $A_i$ of the form \eqref{eq:irr-not-strong-irr}, and let $E$ be the corresponding self-affine set. Suppose that:
\begin{enumerate}
\item All coefficients of $A_i$ and $v_i$ are algebraic.
\item For any $n\in\mathbb{N}$ and any $\ii\neq \jj\in\{1,\ldots,M\}^n$ such that $T_{\ii}$ and $T_{\jj}$ have the same orientation, $P_x T_{\ii}(0)\neq P_x T_{\jj}(0)$ and $P_y T_{\ii}(0)\neq P_y T_{\jj}(0)$.
\item The strong open set condition holds.
\end{enumerate}
Then $\dim_H(E) = \adim(T_1,\ldots,T_M)$.
\end{proposition}
\begin{proof}
As usual, we can assume that $\adim(A_1,\ldots,A_M)<2$. Given $\e>0$, let $\Gamma$ and $\nu$ be as given by Proposition \ref{prop:reduction-to-Bernoulli-irr-but-not-strong-irr}. Since hypotheses (1), (2) pass to subsets of iterates, they hold also for the diagonal system $(T_{\mathtt{j}}:\jj\in\Gamma)$.

\begin{figure}\label{fi:three}
    \centering
    \includegraphics[width=0.8\linewidth]{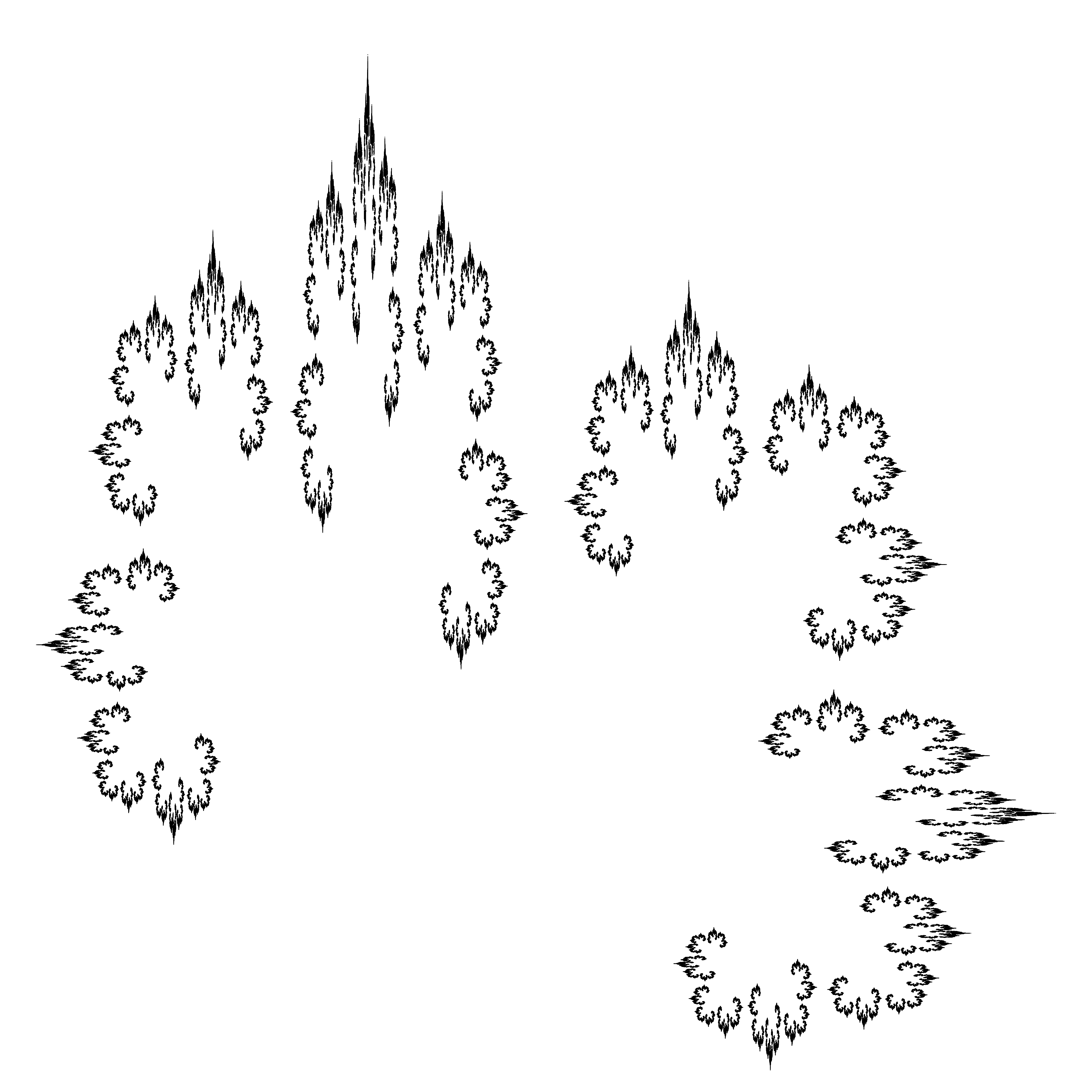}
    \caption{Using Proposition \ref{pr:irr-not-strong-irr-algebraic} it is shown in the article \cite{Mo17} that this self-affine set has Hausdorff dimension equal to the affinity dimension of its defining iterated function system. It is also proved that the affinity dimension $s$ is the unique solution to $\frac{13}{27}\left(\frac{7}{9}\right)^{s-1}+\frac{7}{9}\left(\frac{13}{27}\right)^{s-1}=1$ and is therefore equal to 1.430352022623969408121447296129996697743247230114759\ldots }
    \end{figure}

 Without loss of generality, suppose the horizontal direction corresponds to the largest Lyapunov exponent $\lambda_1(\nu)$.
  Note that $P_x\nu$ is a self-similar measure, whose generating similarities have algebraic coefficients, and such that all finite compositions have different translation parts, thanks to our assumption (2). It then follows from \cite[Theorem 1.1 and Lemma 5.10]{Ho14} that
 \[
 \dim P_x\nu = \min\left( \frac{h(\nu)}{-\lambda_1(\nu)},1 \right).
 \]
 In turn, since the Furstenberg measure is an atom at the horizontal direction, we conclude from \cite[Theorem 2.7]{Ba15} that
 \[
 \dim_H(E) \ge \dim\nu = \adim(T_{\mathtt{j}}:\jj\in\Gamma)) > \adim(T_1,\ldots,T_M)-\e \ge \dim_H(E)-\e.
\]
Since $\e$ was arbitrary, this completes the proof.
\end{proof}
We note that assumptions (2) and (3) in the above proposition are, in general, necessary. Also, the SOSC is weaker than the Rectangular Open Set Condition from \cite{Fr12}.

Similar to the results of \cite{Ho14, Ho15} that we rely on, we can also show that in fairly general parametrized families, there is equality of Hausdorff and affinity dimension outside of a co-dimension $1$ set of parameters.
\begin{proposition}  \label{pr:irr-not-strong-irr-parametrized}
Let $\mathcal{I}\subset\mathbb{R}^p$ be connected and compact. Let $T_i^{(u)}x= A_i^{(u)}x + v_i^{(u)}$, $i=1,\ldots,M$, $u\in \mathcal{I}$ be real-analytic families of invertible affine contractions, where $A_i^{(u)}$ has the form \eqref{eq:irr-not-strong-irr} for all $u\in \mathcal{I}$. Let $E^{(u)}$ be the invariant set for  $T^{(u)}=(T_1^{(u)},\ldots, T_M^{(u)})$. Assume that:
\begin{enumerate}
\item For each pair $\mathtt{i}\neq \mathtt{j}\in\Sigma_M$, neither of the maps
\[
P_x \pi^{(u)}(\mathtt{i}) - P_x \pi^{(u)}(\mathtt{j}), \, P_y \pi^{(u)}(\mathtt{i}) - P_y \pi^{(u)}(\mathtt{j}):\mathcal{I}\to \R
\]
is identically zero, where $\pi^{(u)}$ is the coding map for $T^{(u)}$.
\item The IFS $T^{(u)}$ satisfies the strong open set condition for each $u\in \mathcal{I}$.
\end{enumerate}
Then there exists a set $\mathcal{E}\subset\mathcal{I}$ of Hausdorff and packing dimension at most $p-1$ (in particular, of zero Lebesgue measure), such that
\[
\dim_H(E^{(u)}) = \adim(A_1^{(u)},\ldots, A_M^{(u)}) \quad\text{for all } u\in \mathcal{I}\setminus\mathcal{E}.
\]
\end{proposition}
\begin{proof}
It is enough to show that, given $\e>0$ and a fixed $u\in\mathcal{I}$, there are a neighborhood $\mathcal{I}_u$ of $u\in\mathcal{I}$ and a set $\mathcal{E}_u\subset \mathcal{I}_u$ of Hausdorff and packing dimension at most $p-1$, such that
\begin{equation}  \label{eq:irr-not-strong-irr-parametrized:1}
\dim_H(E^{(u')}) > \adim(T^{(u')})-3\e \quad\text{for all } u'\in \mathcal{I}_u\setminus\mathcal{E}_u.
\end{equation}
Once fixed $\e>0$ and $u\in\mathcal{I}$, let $n$ and $\Gamma$ be as given by Proposition \ref{prop:reduction-to-Bernoulli-irr-but-not-strong-irr} for the IFS $T^{(u)}$ (this assumes that $\adim(A_1^{(u)},\ldots,A_M^{(u)})<2$; the case where the affinity dimension equals $2$ is simpler; details are left to the reader). Let $\lambda_i^{(u)}(\nu)$ be the Lyapunov exponents of $\nu$ with respect to $A_\Gamma^{(u)}=(A_{\mathtt{i}}^{(u)}:\mathtt{i}\in\Gamma)$. Since the maps $A_{\mathtt{i}}^{(u)},\ii\in\Gamma$ are diagonal (by continuity) and $\nu$ is Bernoulli, the maps $u\mapsto \lambda_i^{(u)}(\nu)$ are continuous. The map $u\mapsto \adim(A^{(u)})$ is also continuous, see \cite[Theorem 1.2]{FeSh14}. Hence, there exists a neighborhood $\mathcal{I}_u$ of $u$ in $\mathcal{I}$, such that
\begin{equation} \label{eq:irr-not-strong-irr-parametrized:2}
\ldim(\nu, A_\Gamma^{(u')}) > \ldim(\nu, A_\Gamma^{(u)})-\e > \adim(A^{(u)})-2\e > \adim(A^{(u')})-3\e
\end{equation}
for all $u'\in\mathcal{I}_u$. By making $\mathcal{I}_u$ smaller if needed, we can also assume without loss of generality that the top Lyapunov exponent  $\lambda_1^{(u')}(\nu)$ corresponds to the horizontal direction for all $u'\in\mathcal{I}_u$. As in the proof of Proposition \ref{pr:irr-not-strong-irr-algebraic}, the measures $P_x\pi^{(u')}\nu$ are self-similar. Moreover, invoking \cite[Theorem 1.10]{Ho15}, we obtain a set $\mathcal{E}_u\subset\mathcal{I}_u$ of packing (and Hausdorff) dimension at most $p-1$, such that
\[
\dim P_x\pi^{(u')}\nu = \frac{h(\nu)}{\lambda_1^{(u')}(\nu)}  \quad\text{for all }u'\in\mathcal{I}_u\setminus \mathcal{E}_u.
\]
Applying \cite[Theorem 2.7]{Ba15} as in the proof of Proposition \ref{pr:irr-not-strong-irr-algebraic}, we conclude that
\[
\dim_H(E^{(u')})\ge \dim(\pi^{(u')}\nu) = \ldim(\nu, A_\Gamma^{(u')})   \quad\text{for all }u'\in\mathcal{I}_u\setminus \mathcal{E}_u.
\]
Together with \eqref{eq:irr-not-strong-irr-parametrized:2}, this establishes \eqref{eq:irr-not-strong-irr-parametrized:1}, finishing the proof.
\end{proof}
The first assumption in the above proposition is very mild: it roughly says that the principal projections do not have overlaps ``built-in'' for all parameters.

In some cases, it may be possible to remove any separation assumptions in Propositions \ref{pr:irr-not-strong-irr-algebraic} and \ref{pr:irr-not-strong-irr-parametrized} by using the results from \cite{BaRaSi15}, but we do not pursue this.



\end{document}